\documentclass[11pt]{article}
\pdfoutput=1

\usepackage[utf8]{inputenc}
\usepackage[T1]{fontenc}
\usepackage[english]{babel}
\usepackage[protrusion=true,expansion=true]{microtype}

\usepackage{mathrsfs, amsmath, amssymb, amsthm, graphicx}
\usepackage[bookmarksnumbered]{hyperref}

\usepackage{multicol,color,longtable,bm}
\usepackage{currfile}
\usepackage[a4paper]{geometry}
\usepackage[usenames,dvipsnames]{xcolor}

\usepackage[parfill]{parskip}
\usepackage{topcapt}
\usepackage[font=footnotesize,labelfont=bf]{caption}
\usepackage{hhline}
\usepackage{mdframed}

\usepackage{mathtools}
\usepackage{xparse}
\usepackage{tabu}
\usepackage{booktabs}

\usepackage{cite}

\hypersetup{
    pdftitle={ Small automorphic representations and degenerate Whittaker vectors },         %
    pdfauthor={ Henrik P. A. Gustafsson, Axel Kleinschmidt and Daniel Persson },      %
    linktocpage=true,          
    colorlinks=true,            
    linkcolor=blue!75!black,    
    citecolor=green!50!black,   
    filecolor=magenta,          
    urlcolor=blue!75!black      
}

\usepackage{tikz}
\usetikzlibrary{backgrounds, decorations.markings, arrows, calc}

\tikzset{->-/.style={decoration={markings, mark=at position 0.5 with {\draw (-2pt,-4pt)--(2pt,0)--(-2pt,4pt);}}, postaction={decorate}}}
\tikzset{-<-/.style={decoration={markings, mark=at position 0.5 with {\draw (2pt,-4pt)--(-2pt,0)--(2pt,4pt);}}, postaction={decorate}}}
\tikzstyle{dot}=[circle, draw, thick, fill=white]
\tikzstyle{line}=[thick, shorten >=-2pt, shorten <=-2pt]
\tikzstyle{doubleline}=[thick, double distance=2pt, shorten >=-2pt, shorten <=-2pt]
\tikzstyle{tripleline}=[thick, double distance=4pt, shorten >=-2pt, shorten <=-2pt]

\newtheoremstyle{italics-style}{\parskip}{0}{\itshape}{}{\bfseries}{.}{ }{\thmname{#1}\thmnumber{ #2}\thmnote{ (#3)}}

\newtheoremstyle{normal-style}{\parskip}{0}{}{}{\bfseries}{.}{ }{\thmname{#1}\thmnumber{ #2}\thmnote{ (#3)}}

\newtheoremstyle{named}
{\parskip}
{0}
{\itshape}
{}
{\bfseries}
{.}
{ }
{\thmname{#3}}

\theoremstyle{named}
\newtheorem*{namedtheorem*}{Theorem}

\theoremstyle{italics-style}
\newtheorem{maintheorem}{Theorem}

\newtheorem{theorem}{Theorem}[section]
\newtheorem{lemma}{Lemma}[section]
\newtheorem{proposition}{Proposition}[section]

\newtheorem{definition}{Definition}[section]

\theoremstyle{normal-style}
\newtheorem{remark}{Remark}[section]
\newtheorem{example}{Example}[section]

\newcommand{\nn}{\nonumber}

\newcommand{\rats}{\mathbb{Q}}

\newcommand{\ads}{\mathbb{A}}

\newcommand{\cx}{\mathbb{C}}

\DeclareMathAlphabet{\mathbbold}{U}{bbold}{m}{n}

\newcommand{\id}{\mathbbold{1}}

\newcommand{\lie}[1]{{\mathfrak{#1}}}

\newcommand{\beq}{\begin{equation}}
\newcommand{\eeq}{\end{equation}}
\newcommand{\beqa}{\begin{eqnarray}}
\newcommand{\eqa}{\end{eqnarray}}

\newcommand{\M}{M}

\newcommand{\bs}{\backslash}

\NewDocumentCommand{\intl}{o}{
    \IfNoValueTF{#1}{\int\limits}{\int\limits_{\mathclap{#1}}}
}
\NewDocumentCommand{\suml}{o}{
    \IfNoValueTF{#1}{\sum\limits}{\sum\limits_{\mathclap{#1}}}
}

\newcommand{\Oh}{\mathcal{O}}

\DeclareMathOperator{\tr}{tr}
\DeclareMathOperator{\rank}{rank}
\DeclareMathOperator{\GKdim}{GKdim}
\DeclareMathOperator{\Ind}{Ind}
\DeclareMathOperator{\diag}{diag}
\DeclareMathOperator{\e}{\mathbf{e}}

\numberwithin{equation}{section}

\begin{document}

\thispagestyle{empty}

{\flushright{\tt{AEI-2014-064}\\}}
\vspace{10mm}

\begin{center}
{\LARGE \bf Small automorphic representations\\[2mm] and degenerate Whittaker vectors}\\[10mm]

\vspace{8mm}
\normalsize
{\large Henrik P. A. Gustafsson${}^{1}$, Axel Kleinschmidt${}^{2,3}$ and Daniel Persson${}^1$}

\vspace{10mm}

 ${}^1${\it Chalmers University of Technology, Dept. of Fundamental Physics\\
SE-412\,96 Gothenburg, Sweden}\\
\texttt{\footnotesize{henrik.gustafsson, daniel.persson@chalmers.se}}
\vskip 1 em
${}^2${\it Max-Planck-Institut f\"{u}r Gravitationsphysik (Albert-Einstein-Institut)\\
Am M\"{u}hlenberg 1, DE-14476 Potsdam, Germany}\\
\texttt{\footnotesize{axel.kleinschmidt@aei.mpg.de}}
\vskip 1 em
${}^3${\it International Solvay Institutes\\
ULB-Campus Plaine CP231, BE-1050 Brussels, Belgium}
\vspace{1cm}

\vspace{5mm}

\hrule

\vspace{3mm}

\begin{tabular}{p{12cm}}
{\footnotesize We investigate Fourier coefficients of automorphic forms on split simply-laced Lie groups $G$. We show that for 
automorphic representations of small Gelfand-Kirillov dimension the Fourier coefficients are 
completely determined by certain degenerate Whittaker vectors on $G$. Although we expect our results to hold for arbitrary simply-laced groups, 
we give complete proofs only for $G=SL(3)$ and $G=SL(4)$. This is based on a method of Ginzburg that associates Fourier coefficients of 
automorphic forms with nilpotent orbits of $G$. Our results  complement and extend recent results of Miller and Sahi. We also use our formalism 
to calculate various local (real and $p$-adic) spherical vectors of minimal representations of the exceptional groups $E_6, E_7, E_8$ using global (adelic) degenerate Whittaker vectors, 
correctly reproducing existing results for such spherical vectors obtained by very different methods. }
\end{tabular}
\vspace{5mm}
\hrule
\end{center}

\newpage
\setcounter{page}{1}

\tableofcontents

\pagebreak

\section{Introduction}
This paper is concerned with Fourier coefficients of automorphic forms on split, simply-laced Lie groups $G$, attached to certain special unipotent (in the sense of Arthur~\cite{Arthur1,Arthur2}) representations of unusually small functional (or Gelfand-Kirillov) dimension. In the modern theory of automorphic forms one usually considers $G$ to be the group of adelic points $G(\mathbb{A}_k)$ of an algebraic group $G(k)$ over some number field $k$. In this paper we  restrict to the case $k=\mathbb{Q}$ and write $\mathbb{A}$ for the adeles of $\mathbb{Q}$. 

The study of Fourier coefficients has been at the centre of attention in the theory of automorphic forms for many years. A major application lies in Langlands' theory of automorphic $L$-functions and the transfer between automorphic representations of a group $G(\mathbb{A})$ to another $G'(\mathbb{A})$. According to Langlands one can attach an $L$-function $L(\pi, s)$ to each automorphic representation $\pi$ of $G$, depending on a complex variable $s$. A basic conjecture is that this function admits a meromorphic continuation to all of $\mathbb{C}$ and satisfies a functional equation relating $L(\pi, s)$ with $L(\pi, 1-s)$. Roughly, the principle of functoriality then asserts that if there is a functorial transfer from an automorphic representation $\pi$ of $G(\mathbb{A})$ to $\pi'$ of $G'(\mathbb{A})$ then the associated $L$-functions are equal $L(\pi, s)=L(\pi', s)$. Langlands showed \cite{LanglandsEP} that a rich source of $L$-functions arises from the constant terms (i.e. zeroth Fourier coefficients) of Eisenstein series, and Shahidi extended this method \cite{Shahidi} and showed that also the non-constant Fourier coefficients give rise to automorphic $L$-functions.

A complete Fourier expansion of an automorphic form $\varphi$ on a group $G(\mathbb{A})$ is in general hard to come by. The idea is to choose a unipotent subgroup $U\subset G$ and try to write $\varphi$ as a sum of terms $\sum_{\psi_U} F_{\psi_U}$, where the sum is over unitary characters $\psi_U$ on $U(\mathbb{A})$ trivial on $U(\mathbb{Q})$ and each ``Fourier coefficient'' $F_{\psi_U}$ is manifestly invariant with respect to the discrete subgroup $U(\mathbb{Q})\subset U(\mathbb{A})$. In effect, one wishes to diagonalise the action of $U(\mathbb{Q})$. This works well if $U$ is abelian but whenever $U$ is non-abelian the expansion is considerably more complicated. However, it may happen that for  special types of automorphic representations the task of obtaining the Fourier expansion simplifies due to the fact that many of the Fourier coefficients vanish because of representation-theoretic constraints on the function space. A simple instance of this phenomenon occurs in the classical case of holomorphic modular forms for $SL(2, \mathbb{Z})$. These are $SL(2,\mathbb{Z})$-covariant functions $\varphi(\tau)$ on the complex upper half plane $SL(2,\mathbb{R})/SO(2)$ which admit a Fourier expansion of the form $\varphi(\tau)=\sum_{n\geq 0} a(n) e^{2\pi i n \tau}$, where the sum is restricted to $n\geq 0$ in order for the function to be holomorphic in $\tau$. Equivalently, one can view this representation-theoretically and say that all coefficients $a(n)$ must vanish whenever $n<0$ due to the fact that $\varphi$ is attached to the holomorphic discrete series of $SL(2,\mathbb{R})$.  

Similar phenomena may occur for higher-rank groups when restricting to automorphic representations $\pi$ with small Gelfand-Kirillov dimension. The typical example is the \emph{minimal representation} $\pi_{min}$ which has the smallest non-trivial Gelfand-Kirillov dimension among all representations \cite{MR0404366,MR1159103,Gunaydin:2001bt}. For simply-laced groups $G$ it was shown in the seminal paper by Ginzburg-Rallis-Soudry \cite{GRS} that automorphic forms attached to $\pi_{min}$ have ``very few'' non-vanishing Fourier coefficients. Automorphic forms attached to the minimal representation can be realised as special points in the parameter space of Langlands-Eisenstein series and this fact has  been used extensively to study liftings of automorphic forms via the so called \emph{theta correspondence} (see, e.g., \cite{Prasad}).

Fourier coefficients of automorphic forms also play an important role in string theory, where they capture certain non-perturbative (instanton) effects in gravitational scattering amplitudes and black hole partition functions. In recent years there has been a lot of progress in understanding the relation to automorphic representations and it turns out that special unipotent automorphic representations, like the minimal representation, show up naturally in this context \cite{Green:2010kv,GreenSmallRep,Pioline:2010kb,Fleig:2012zc,FK2012,FKP,Bossard:2014lra,Bossard:2014aea}. In \cite{GreenSmallRep} the results of Ginzburg-Rallis-Soudry were  extended to automorphic forms in the next-to-minimal representation $\pi_{ntm}$ which were also shown to have very few Fourier coefficients, a fact that has a direct interpretation in string theory. Fourier coefficients attached to small automorphic representations have also been proposed to capture  microscopic degeneracies of certain black holes in string theory; see for instance \cite{Pioline:2005vi,Gunaydin:2005mx,Pioline:2009qt,Bao:2009fg,Bao:2010cc,PerssonAuto}. 

The discussion above was phrased in the global language of automorphic representations $\pi$ of adelic groups $G(\mathbb{A})$. For so called \emph{admissible representations} one has an Euler product factorisation $\pi=\otimes \pi_p$ into local representations $\pi_p$ for each prime $p\leq \infty$. For finite primes $p<\infty$ these are $p$-adic (or non-Archimedean) representations of $G(\mathbb{Q}_p)$ while for $p=\infty$ these are real (or Archimedean) representations of $G(\mathbb{R})$. There is a corresponding notion of minimal representation $\pi_{min, p}$ also for local representations and the analogues of the global Fourier coefficients attached to $\pi_{min, p}$ are so called {\it $p$-adic spherical vectors} $f_p^{\circ}$ which are vectors in $\pi_{min, p}$ invariant under the compact subgroup $G(\mathbb{Z}_p)\subset G(\mathbb{Q}_p)$. Representations for which such vectors exist are called \emph{unramified}. For minimal representations it is rather difficult to obtain explicit expressions for these spherical vectors, but results are in particular known for various realisations of the minimal representation of exceptional groups $E_6, E_7, E_8$ \cite{DS,Kazhdan:2001nx,KazhdanPolishchuk,SavinWoodbury}. 

The simplest type of Fourier coefficient $F_{\psi_U}$ occurs when $U$ is taken to be the maximal unipotent radical $N$ in the Levi decomposition $B=AN$ of the standard Borel subgroup $B\subset G$. This is known as a \emph{Whittaker coefficient}, usually denoted by $W$, and it is a famous result of Langlands that this factorises $W=\otimes W_p$ into a product of local Whittaker coefficients $W_p$ for each local representation $\pi_p$. In general, however, a global Fourier coefficient $F_{\psi_U}$ of an automorphic form $\varphi\in \pi$ does not exhibit a similar Euler product factorisation since it is typically given by a non-trivial sum of different Weyl orbits that does not factorise.  As we shall see, though, in the case when $\pi$ is a minimal representation it can happen that $F_{\psi_U}$  factorises also for unipotent radicals $U$, other than the maximal one $N$, and the local factors $F_{\psi_U,p}$ are then given by the spherical vectors $f_p^{\circ}$ discussed above. 

Given an automorphic representation $\pi$ an important question is to determine whether a certain Fourier coefficient vanishes or not. A powerful method for doing this has been developed by Ginzburg \cite{MR2214128} and Miller-Sahi \cite{MillerSahi}, following earlier local results of Moeglin-Waldspurger \cite{MR913667} (for $p<\infty$) and Matumoto \cite{MR892192} (for $p=\infty$). The idea is to parametrize the Fourier coefficients by nilpotent $G$-orbits, i.e. the adjoint action of $G$ on any nilpotent element in the Lie algebra $\mathfrak{g}$. Each automorphic representation $\pi$ with Gelfand-Kirillov dimension $n$ can itself be associated with a nilpotent orbit $\mathcal{O}$ (sometimes denoted $\Oh_\pi$) of dimension $2n$ via Kirillov's ``orbit method''~\cite{MR1701415}. Each Fourier coefficient of $\pi$ associated with nilpotent orbits which are outside of the closure $\overline{\mathcal{O}}$ is then expected to vanish. We refer to the set of nilpotent orbits with non-vanishing Fourier coefficients as the \textit{wavefront set}. For example, when $\pi$ is the minimal representation $\pi_{min}$, it is known that only those Fourier coefficients attached to the trivial orbit $\mathcal{O}_0$ and the minimal orbit $\mathcal{O}_{min}$ are non-vanishing \cite{GRS} and the wavefront set is given by the closure of the minimal nilpotent orbit with Bala--Carter label $A_1$.\footnote{See section~\ref{nilpotentorbits} for a brief review of nilpotent orbits and their labelling.}

Miller and Sahi \cite{MillerSahi} use a related but slightly different perspective based on the Piatetski-Shapiro-Shalika method~\cite{MR546599,MR0348047} to show that any Fourier coefficients of the minimal representations of the exceptional groups $E_6, E_7$ are completely determined by certain maximally degenerate Whittaker vectors. 

In this paper we combine the results of Miller-Sahi with the method of Ginzburg to establish various results concerning Fourier coefficients of certain Eisenstein series attached to special automorphic representations. We will consider both the global and the local perspective. Our main interest is eventually with the exceptional groups $E_6, E_7, E_8$ but many of our results concern the case of $SL(n)$ for $n=3,4$. Let us now briefly summarise the main results of the paper. 

Let $G(\mathbb{Q})$ be a split, simply-laced Lie group and $G(\mathbb{A})$ its adelization. Let $B=AN$ be the Borel subgroup, and introduce a quasi-character $\chi : B(\mathbb{Q})\backslash B(\mathbb{A})\to \mathbb{C}^*$. Let $E(\chi, g)$ be the associated Langlands-Eisenstein series on $G(\mathbb{A})$, attached to the non-unitary principal series $\text{Ind}_{B}^{G}\chi$. The minimal representation $\pi_{min}$ and next-to-minimal representation $\pi_{ntm}$ of $G(\mathbb{A})$ can both be realised as submodules of $\text{Ind}_B^G\chi$ for certain choices of $\chi$ \cite{GRS,GreenSmallRep}. We then have 

\begin{maintheorem}
    \label{thm:orbit-exp}
    Let $\varphi$ be an automorphic form on the special linear group $G(\ads) = SL(3, \ads)$ or $G(\ads) = SL(4, \ads)$ belonging to the principal series $\Ind_{B}^{G}\chi$. Then, $\varphi$ can be expanded as
   \begin{equation}
       \varphi(\chi, g) = \sum_{\Oh} \mathcal F_\Oh(\chi, g) 
   \end{equation} 
   where the sum is over all nilpotent orbits $\Oh$ of $G$. Each $\mathcal F_\Oh$ is (linearly) determined by Fourier coefficients $F_\Oh$ (see section \ref{sec:orbit-rep}) attached to the nilpotent orbit $\Oh$.
    If $\varphi$ belongs to an automorphic subrepresentation $\pi$ of $\Ind_{B}^{G}\chi$ with associated nilpotent orbit $\Oh_\pi$, then all $\mathcal F_\Oh$ where $\Oh \notin \overline{\Oh_\pi}$ vanish. 
\end{maintheorem}

Note that there are a finite number of terms $\mathcal F_\Oh$, but that each $\mathcal F_\Oh$ contains a infinite number of Fourier coefficients with different characters. Furthermore, there is some ambiguity in how each separate $\mathcal F_\Oh$ is defined for different representations. See the discussion after the proof of theorem \ref{thm:SL4-min-rep} for more details. We will prove this theorem using the Eisenstein series $E(\chi,g)$ that is the spherical vector in the automorphic realisation of the minimal representation but since we are not using sphericality there will be no loss of generality in the proof. 

\begin{remark}
Theorem \ref{thm:orbit-exp} is very reminiscent of the expansion of the character distribution of a representation $\pi$ when restricted to a sufficiently small neighbourhood of zero. More precisely, according to Harish-Chandra \cite{MR1702257} and Howe \cite{MR0342645} the character $\chi_\pi$ of a $G$-representation $\pi$ has an expansion
\beq
\chi_\pi = \sum_{\mathcal{O}} c_\mathcal{O}(\pi) \widehat{\mu}_\mathcal{O},
\eeq
where the sum runs over all nilpotent $G$-orbits, $\widehat{\mu}_\mathcal{O}$ denotes the Fourier transform of a distribution on $G$ (a certain ``orbital integral''), and  $c_\mathcal{O}(\pi)$ are complex numbers. For the trivial representation all the $\widehat{\mu}_\mathcal{O}$ vanish except the one associated with the trivial orbit $\mathcal{O}_0$. For the $p$-adic minimal representation $\pi_{min, p}$ one has the following fundamental result (see \cite{GanSavin}):
\begin{equation}
\chi_{\pi_{min, p}} = \widehat{\mu}_{\mathcal{O}_{min}}+ c_0
\end{equation}
where $\mathcal{O}_{min}$ is the smallest non-trivial nilpotent orbit (minimal orbit).\footnote{This formula can in fact be taken as the definition of $\pi_{min, p}$.} The structure of $\chi_\pi$ is closely related to Fourier coefficients of automorphic forms by the following results. An admissible automorphic representation $\pi$ has a Whittaker model only if there is a regular nilpotent orbit $\mathcal{O}$ such that $c_{\mathcal{O}}(\pi)$ is non-zero \cite{MR0393355}. Moreover, it was shown by Moeglin-Waldspurger \cite{MR913667} in the $p$-adic setting that for the maximal orbit $\mathcal{O}$ such that $c_\mathcal{O}(\pi)$ is non-zero, there exists a (possibly degenerate) Whittaker model whose dimension is precisely $c_\mathcal{O}$. For example, for generic representations $\pi$ (e.g. the full principal series) the leading orbit is the regular orbit $\mathcal{O}_{reg}$ and thus in this case $c_\mathcal{O}(\pi)$ gives the dimension of the generic Whittaker model of $\pi$. In the real setting (i.e. for $p=\infty$) similar results were obtained by Matumoto \cite{MR892192}. Relations between degenerate Whittaker vectors and small representations have also been explored in more recent work by Gourevitch and Sahi \cite{MR3029948,2012arXiv1210.4064G}. In general, very little is known about the numbers $c_{\mathcal{O}}$ and for intermediate representations one can have situations with non-trivial multiplicities of Whittaker models, i.e. when $c_\mathcal{O}(\pi)> 1$.\footnote{We thank Gordan Savin for very helpful correspondence on these issues.}  
\end{remark}

\vspace{.4cm} 

The following theorem then shows that, in the minimal representation, $\varphi$ is completely determined by maximally degenerate Whittaker vectors (defined in section \ref{DegWhit}) and the constant term attached to the trivial orbit.

\begin{maintheorem}
    \label{thm:min-rep}
    Let $\varphi$ be an automorphic form on the special linear group $G(\ads)=SL(3, \ads)$ or $G(\ads) = SL(4, \ads)$ belonging to the minimal representation $\pi_\text{min}$. All Fourier coefficients $F_\Oh$ of $\varphi$ attached to nilpotent orbits outside of the closure $\overline{\Oh_\text{min}}$ of the minimal orbit $\Oh_\text{min}$ then vanish, and those attached to $\Oh_\text{min}$ are completely determined by the maximally degenerate Whittaker vectors 
    \begin{equation}
        W_{\psi_\alpha}(g)=\intl[N(\rats)\bs N(\ads)] \varphi(ng)\overline{\psi_\alpha(n)}dn, 
        \label{maxdegwhit}
    \end{equation}
    where $\psi_\alpha$ is non-trivial only on a one-parameter subgroup $N_\alpha\subset N$ corresponding to a single simple root $\alpha$ of $\mathfrak{g}$.
\end{maintheorem}

The explicit expression for $F_{\Oh_\text{min}}$ in the minimal representation in terms of such Whittaker vectors can be found in the proof of the theorem in section \ref{sec:SL3} for $SL(3)$ and section \ref{sec:SL4} for $SL(4)$. 

We define the next-to-minimal representation to be that automorphic representation with wavefront set given by the next-to-minimal orbit, i.e., the one given by Bala--Carter label $2A_1$. Then we have the following result for the next-to-minimal representation:

\begin{maintheorem}
    \label{thm:ntm-rep}
    Let $\varphi$ be an automorphic form on the special linear group $G(\ads) = SL(4, \ads)$ belonging to the next-to-minimal representation $\pi_\text{ntm}$. The closure $\overline{\Oh_\text{ntm}}$ of the next-to-minimal nilpotent orbit contains $\Oh_\text{ntm}$, $\Oh_\text{min}$ and $\Oh_0$. All Fourier coefficients $F_\Oh$ of $\varphi$ attached to nilpotent orbits outside $\overline{\Oh_\text{ntm}}$ vanish; coefficients attached to $\Oh_\text{ntm}$ are completely determined by the degenerate Whittaker vectors 
    \begin{equation}
        \label{nextmaxdegwhit}
        W_{\psi_{\alpha, \beta}}(\chi, g)=\int_{N(\rats)\backslash N(\ads)} \varphi(ng)\overline{\psi_{\alpha, \beta}(n)} \, dn, 
    \end{equation}
    where $\psi_{\alpha, \beta}$ is non-trivial only on a two-parameter subgroup $N_{\alpha, \beta}$ corresponding to two commuting simple roots $\alpha, \beta$ of $\mathfrak{g}$.
    Coefficients attached to $\mathcal{O}_{min}$ are completely determined by degenerate Whittaker vectors of the form \eqref{maxdegwhit} and \eqref{nextmaxdegwhit}.
\end{maintheorem}

Note that $SL(3,\ads)$ does not have a nilpotent orbit of type $2A_1$ which is why the above theorem is only stated for $SL(4, \ads)$. (The next largest orbit for $SL(3,\ads)$ after the minimal $A_1$-type orbit is the regular $A_2$-type orbit~\cite{CollingwoodMcGovern}.)

The explicit expressions for $F_{\Oh_\text{min}}$ and $F_{\Oh_\text{ntm}}$ in the next-to-minimal representation in terms of the Whittaker vectors above can be found in the proof of the theorem in section \ref{sec:SL4}. As a result, in the next-to-minimal representation, $\varphi$ is completely determined by maximally degenerate Whittaker vectors and the degenerate Whittaker vectors of the form \eqref{nextmaxdegwhit} together with the constant term.

\begin{remark}
    Theorem \ref{thm:min-rep} can be viewed as a concretization of the main result in Miller-Sahi \cite{MillerSahi} but for $SL(3, \ads)$ and $SL(4, \ads)$, while theorem \ref{thm:ntm-rep} is the generalisation of this result to the next-to-minimal representation. We expect that theorem \ref{thm:ntm-rep} generalises to any split, simply-laced Lie group $G$. 
\end{remark}

Although at the moment we cannot prove the above results for the exceptional Lie groups $E_6, E_7, E_8$ we can still use the same philosophy to  rederive local results about spherical vectors for minimal representations from the explicit knowledge of global degenerate Whittaker vectors obtained in our previous work \cite{FKP}. The minimal representations $\pi_{min}$ factorises according to 
\beq
\pi_{min}=\bigotimes_{p\leq \infty} \pi_{min, p}
\eeq
and so for any vector $f\in \pi_{min}$ we have $f=\otimes_p f_p$ with $f_p\in \pi_{min, p}$ (including $p=\infty$). This fact allows us to use very simple manipulations to deduce known results, which were obtained by very different methods.

Let $G$ be one of the Lie groups $E_6, E_7, E_8$ and $\mathfrak{g}$ one of the associated Lie algebras $\mathfrak{e}_6, \mathfrak{e}_7, \mathfrak{e}_8$. Let $\varphi\in \pi_{min}$ be a vector in the minimal representation of $G(\mathbb{A})$, obtained by restricting the character $\chi$ in the Eisenstein series $E(\chi, g)$ to the special value $\chi=\chi_0$. Let $W_{\psi_\alpha}(\chi, g)$ be the associated maximally degenerate Whittaker vector for $\alpha$ a simple root of $\mathfrak{g}$. These Whittaker vectors were calculated for all simple roots in Appendices A.1, A.2 and A.3 of \cite{FKP}. In what follows we use the Bourbaki labelling of the simple roots of \nolinebreak[4]$\lie g$.

Any global Whittaker vectors $W_{\psi}$ factorises into local components as $W_\psi=\prod_{p\leq \infty} W_{\psi_p}$ and one can compare the various local components to spherical vectors of the minimal representation already computed in the literature. We have the results:

\begin{proposition}[the case of $G=E_6$] 
\label{prop:E6}\mbox{} 
\begin{itemize}
\item[(i)] The Archimedean component of  $W_{\psi_{\alpha_1}}(\chi, g)$ is equal to the real spherical vector $f_\infty^{\circ}$ in the minimal representation $\pi_{min, \infty}$ of $E_6(\mathbb{R})$ computed by Dvorsky-Sahi \cite{DS}; 

\item[(ii)] the non-Archimedean ($p$-adic) component of $W_{\psi_{\alpha_1}}(\chi, g)$ is equal to the $p$-adic spherical vector $f_p^{\circ}$ in the minimal representation $\pi_{min, p}$ of $E_6(\mathbb{Q}_p)$ computed by Savin-Woodbury \cite{SavinWoodbury}; 

\item[(iii)] the Archimedean component of $W_{\psi_{\alpha_2}}(\chi, g)$ is equal to the (abelian part of) the real spherical vector in the Heisenberg realisation of the minimal representation $\pi_{min, p}$ of $E_6(\mathbb{R})$, as computed by Kazhdan-Pioline-Waldron \cite{Kazhdan:2001nx};

\item[(iv)] the non-Archimedean component of $W_{\psi_{\alpha_2}}(\chi, g)$ is equal to the $p$-adic spherical vector in the Heisenberg realisation of the minimal representation $\pi_{min, p}$ of $E_6(\mathbb{Q}_p)$, computed by Kazhdan-Polishchuk \cite{KazhdanPolishchuk}.
\end{itemize}
\end{proposition}

\begin{proposition}[the case of $G=E_7$] 
\label{prop:E7}
\mbox{}
\begin{itemize}

\item[(i)] the Archimedean component of  $W_{\psi_{\alpha_7}}(\chi, g)$ is equal to the real spherical vector $f_\infty^{\circ}$ in the minimal representation $\pi_{min, \infty}$ of $E_7(\mathbb{R})$ computed by Dvorsky-Sahi \cite{DS}; 

\item[(ii)] the non-Archimedean component of $W_{\psi_{\alpha_7}}(\chi, g)$ is equal to the $p$-adic spherical vector in the minimal representation $\pi_{min, p}$ of $E_7(\mathbb{Q}_p)$, as computed by Savin-Woodbury~\cite{SavinWoodbury}; 

\item[(iii)] the Archimedean component of  $W_{\psi_{\alpha_1}}(\chi, g)$ is equal to the (abelian part of) the real spherical vector in the Heisenberg realisation of the minimal representation $\pi_{min, p}$ of $E_7(\mathbb{R})$, as computed by Kazhdan-Pioline-Waldron \cite{Kazhdan:2001nx}; 

\item[(iv)] the non-Archimedean component of $W_{\psi_{\alpha_1}}(\chi, g)$ is equal to the $p$-adic spherical vector in the Heisenberg realisation of the minimal representation $\pi_{min, p}$ of $E_7(\mathbb{Q}_p)$, computed by Kazhdan-Polishchuk \cite{KazhdanPolishchuk}.
\end{itemize}
\end{proposition}

\begin{proposition}[the case of $G=E_8$]
\label{prop:E8}
 \mbox{}  
\begin{itemize}

\item[(i)] the Archimedean component of $W_{\psi_{\alpha_8}}(\chi, g)$ is equal  to (the abelian part of) the real spherical vector in the Heisenberg realisation of the minimal representation $\pi_{min, \infty}$ of $E_8(\mathbb{R})$, as computed by Kazhdan-Pioline-Waldron \cite{Kazhdan:2001nx}; 

\item[(ii)] the non-Archimedean component of $W_{\psi_{\alpha_8}}(\chi, g)$ is equal to the $p$-adic spherical vector in the Heisenberg realisation of the minimal representation $\pi_{min, p}$ of $E_8(\mathbb{Q}_p)$, computed by Kazhdan-Polishchuk \cite{KazhdanPolishchuk}. 
\end{itemize}
\end{proposition}

\noindent These propositions are proven in section \ref{sec_proofprop}.

Our results also have applications to the study of instanton effects in gravitational scattering amplitudes in string theory, and this will be studied in a follow-up publication.

\textit{Acknowledgements.} We are especially indebted to David Ginzburg for extensive correspondence on his method of associating Fourier coefficients with nilpotent orbits. Without his explanations this work would have been impossible. Special gratitude is owed to Philipp Fleig who contributed to this work at the early stages. 
We are also grateful to Ben Brubaker, Michael Green, Steve Miller, Manish Patnaik, Boris Pioline, Gordan Savin and Pierre Vanhove for valuable discussions.

\section{Degenerate Whittaker vectors and Fourier expansions}
\label{DegWhit}

\subsection{Fourier coefficients of Eisenstein series}
\label{sec:fourier-coeffs}
Let $G(\mathbb{Q})$ be an algebraic reductive group and $G(\mathbb{A})$ its adelization. For definiteness we restrict to the case when $G$ is semi-simple, simply-laced and split. Let $Q$ be a standard parabolic subgroup with Levi decomposition $Q=MR$, where $M$ is the Levi factor and $R$ is the unipotent radical. Fix a quasi-character $\chi : Q(\mathbb{A})\to \mathbb{C}^*$, trivial on $Q(\mathbb{Q})$, and extend it to all of $G$ by $\chi(mrk)=\chi(mr)$, with $m\in M, r\in R$, and $k$ an element of the maximal compact subgroup $K$. Attached to such a character we have the Langlands-Eisenstein series
\beq
E(\chi, g)=\sum_{\gamma\in Q(\mathbb{Q})\backslash G(\mathbb{Q})} \chi(\gamma g),
\eeq
which converges absolutely on a subspace of the space of $\chi$'s. Representation-theoretically this Eisenstein series is attached to the (degenerate) principal 
series 
\beq
\text{Ind}_{Q}^{G}\chi=\{f:G\to \mathbb{C}\, |\, f(mrg)=\chi(mr)f(g), \forall m\in M, r\in R\}.
\label{principalseries}
\eeq
The group $G$ acts on this space by right-translation: $[\rho(g)\cdot f](h)=f(hg)$ and consequently does not affect the automorphic invariance of $E(\chi, g)$ on the left. Functions in this space are determined by their restriction to the parabolic coset $Q\backslash G$, and hence the \emph{functional} (or \emph{Gelfand-Kirillov}) dimension of $\text{Ind}_{Q}^{G}\chi$ is generically given by the dimension of $Q\backslash G$.

We will be particularly interested in a special class of such Eisenstein series attached to certain $G$-representations of unusually small Gelfand-Kirillov dimension. To this end we let $Q$ be a \emph{maximal} parabolic subgroup of $G$. In this case the character $\chi$ can be parametrised by a single complex number $s$ and we denote it by $\chi_s$. For certain special values of $s$ the Gelfand-Kirillov dimension of $\text{Ind}_{Q}^{G}\chi_s$ reduces, and in particular for some $s=s_0$ it contains the minimal representation $\pi_{min}$ as a submodule. The minimal representation is distinguished by having the smallest possible non-trivial Gelfand-Kirillov dimension among all representations. For the exceptional groups $E_6, E_7, E_8$ this occurs for $s=3/2$ if we choose $Q$ to be the maximal parabolic subgroup parametrised by the first simple root $\alpha_1$ (in Bourbaki labelling) \cite{GRS,GreenSmallRep,Pioline:2010kb}. This implies that for the exceptional groups  $E(3/2, g)$ belongs to $\pi_{min}$. Similarly, one has that $E(5/2, g)$ is an element of the next-to-minimal representation \cite{GreenSmallRep}.

Returning for a moment to the general case of an arbitrary character $\chi$ we shall  now introduce a class of Fourier coefficients of $E(\chi, g)$. To this end let $B=NA$ be the Borel subgroup of $G$ with $A$ the Cartan torus and $N$ the unipotent radical. Let 
\beq
\psi : N(\mathbb{Q})\backslash N(\mathbb{A})\to U(1) 
\eeq
be a unitary character of $N(\mathbb{A})$, trivial on $N(\mathbb{Q})$. This character decomposes according to 
\beq
\psi=\psi_\infty \prod_{p< \infty} \psi_p,
\label{characterfactorization}
\eeq
where $\psi_p$ is a character on $N(\mathbb{Q}_p)$ (trivial on $N(\mathbb{Z}_p))$, while $\psi_\infty$ is a character on $N(\mathbb{R})$ (trivial on $N(\mathbb{Z}))$. Denote by $N_\alpha$ the restriction of $N$ to its one-parameter subgroup generated by $x_\alpha(u)=e^{uE_\alpha}$, $u\in \mathbb{A}$, where $E_\alpha$ is the Chevalley generator associated with the positive root $\alpha \in \Delta_+$. The restriction of $\psi$ to $N_\alpha$ then yields a character $\psi_\alpha : \mathbb{Q}\backslash \mathbb{A}\to U(1)$, with a similar factorisation as in (\ref{characterfactorization}). Any character $\psi$ on $N$ is determined by its values when restricted to the various $N_\alpha$ with $\alpha$ a simple root. We label the set of simple roots by $\Pi\subset \Delta_+\subset \Delta$. Hence, representing an arbitrary element $n\in N/[N,N]$ as $\prod_{\alpha\in \Pi} x_\alpha(u_\alpha)$ with $u_\alpha\in \mathbb{A}$, we can write an arbitrary character $\psi$ on $N$ as
\beq
\psi(n)= \e\left( \sum_{\alpha\in \Pi} m_\alpha u_\alpha\right)
\label{Ncharacter}
\eeq
with $\psi$ being completely determined by the choice of $m_\alpha\in \mathbb{Q}$ for $\alpha\in\Pi$ a simple root. When all $m_\alpha\neq 0$ we say that $\psi$ is \emph{generic}; otherwise it is \emph{degenerate}. The rationals $m_\alpha$ are sometimes called charges because of their physical interpretation; a name we will adopt here for convenience.

In the above equation we have used the notation
\begin{align}
\label{eNotation}
\e(x) = \exp(2\pi i x),
\end{align}
with the understanding that this is to be understood as an Euler product and for $p<\infty$ the local expression is given by $\exp(2\pi i [x]_p)$ with $[x]_p$ denoting the fractional part of $x$ (the class of $x$ in $\mathbb{Q}_p/\mathbb{Z}_p$).

For a character $\psi$ we  define the \textit{Whittaker vector} of $E(\chi, g)$ along $N$:
\beq
\label{Wvector}
W_\psi(\chi, g)=\int_{N(\mathbb{Q})\backslash N(\mathbb{A})} E(\chi, ng)\overline{\psi(n)}dn. 
\eeq
By construction this satisfies 
\beq
\label{Wfunc}
W_\psi(\chi, nak)=\psi(n)W_\psi(a), 
\eeq 
and so is determined by its restriction to the torus $A$. The $K$-invariance is inherited from that of $E(\chi, g)$. Smooth functions on $G$ with this property are associated with the induced representation 
\beq
\text{Ind}_{N}^{G} \psi=\{ f:G\to \mathbb{C} \, |\, f(ng)=\psi(n)f(g), \forall n\in N\}. 
\eeq

The image of the map $Wh : \text{Ind}_{Q}^{G}\chi \to \text{Ind}_{N}^{G} \psi$ is known as a \emph{Whittaker model}. Much is known about how to compute the Whittaker vector $W_\psi(\chi, g)$; in particular it can be shown that for generic $\psi$ it reduces to the following expression
\beq
\label{Wgeneric}
W_\psi(\chi, a)=\int_{N(\mathbb{A})} \chi(w_0 n a)\overline{\psi(n)}dn, 
\eeq
where $w_0$ is the longest element in the Weyl group of $\mathfrak{g}$~\cite{Jacquet1967} (see also \cite{FGKP} for a discussion). This implies that $W$ is \emph{Eulerian}, i.e. factories into an Euler product 
\beq
W_\psi(\chi, g)=\prod_{p\leq \infty}W_{\psi_p}(\chi, g),
\eeq
where each local factor can be computed separately. In fact, for all finite primes $p<\infty$ a complete formula exists, known as the \emph{Casselman-Shalika formula} \cite{CasselmanShalika}.

More generally, one can consider Fourier coefficients of $E(\chi, g)$ associated with parabolic subgroups other than the Borel. To this end pick another standard parabolic subgroup $P$ with  $B\subset P\subset G$ and consider its Levi decomposition $P=LU$ with unipotent radical $U$ and Levi factor $L$. Let $\psi_U : U(\mathbb{A}) \to U(1) $ be a unitary character on $U(\mathbb{A})$, trivial on $U(\mathbb{Q})$. We can now consider the following Fourier coefficient of $E(\chi, g)$: 
\begin{align}
F_{\psi_U}(\chi,g) = \int\limits_{U(\mathbb{Q})\backslash U(\mathbb{A})} E(\chi,ug) \overline{\psi_U(u)}du.
\label{Fcoefficient}
\end{align}
This coefficient is determined by its values on $l\in L$ and satisfies
\begin{align}
F_{\psi_U}(\chi,ulk) = \psi_U(u) F_{\psi_U}(\chi,l),
\end{align}
where the $K$-invariance is inherited from $E(\chi, g)$. In contrast to the case of Whittaker vectors, discussed above, much less is known about explicit formulas for $F_{\psi_U}(\chi,g)$. In particular, in general one does not expect $F_{\psi_U}(\chi,g)$ to be Eulerian since it is generically given by a non-trivial sum over Weyl orbits; hence there is no direct analogue of the Casselman-Shalika formula. One of the main purposes of this article is to show that in certain situations, in particular for special choices of $\chi$ corresponding to small automorphic representations, one can in fact evaluate $F_{\psi_U}$ explicitly. To pave the way for this, let us first make the following well-known, but crucial, observation that can for instance be found in~\cite{MillerSahi}. 

Under the action of an element $\gamma\in L(\mathbb{Q}) = L(\mathbb{A})\cap G(\mathbb{Q})$ we have
\begin{align}
    \label{eq:L-conjugation}
F_{\psi_U}(\chi,\gamma l) &= \int\limits_{U(\mathbb{Q})\backslash U(\mathbb{A})} E(\chi,u\gamma l) \overline{\psi_U(u)}du\nn\\
&= \int\limits_{U(\mathbb{Q})\backslash U(\mathbb{A})} E(\chi,\gamma ul) \overline{\psi_U(\gamma u\gamma^{-1})}du\nn\\
&= F_{\gamma\cdot \psi_U}(\chi,l),
\end{align}
where we have used the invariance of $du$ under the $L(\mathbb{Q})$ action and defined the transformed character by 
\beq
\big(\gamma\cdot \psi_U\big)(u):=\psi_U(\gamma u \gamma^{-1}).
\eeq
This implies that  the Fourier coefficients come in $L(\mathbb{Q})$-orbits acting on the space of characters $\psi_U$. This space can be identified with a subspace of $\mathfrak{u}^*$, the dual of the nilpotent Lie algebra $\mathfrak{u}$ of $U$. 

Using this identification, we may relate the different $L$-orbits to $G$-orbits of nilpotent elements. To this end we shall now recall some basic facts about nilpotent orbits that will play an important role in what follows.

\subsection{Basics of nilpotent orbits}
\label{nilpotentorbits}

In this section, we review some standard facts about nilpotent orbits of a complex semisimple Lie algebra. Two standard references on the subject are the books by Collingwood-McGovern \cite{CollingwoodMcGovern} and Carter \cite{Carter}. 

Let $G$ be a complex semisimple Lie group and $\mathfrak{g}$ its associated Lie algebra. A nilpotent orbit of $G$ is the orbit of any nilpotent element $X\in \mathfrak{g}$ under the adjoint action of $G$:
\begin{align}
\mathcal{O}_X = \{g X g^{-1} \big| \forall g\in G\}.
\end{align}
If one is interested in orbits under the action of a real or discrete subgroup of the complex Lie group, these orbits typically will have to be further subdivided into finer orbits.

According to the Jacobson-Morozov theorem, one can associate to each orbit $\mathcal{O}_X$ a triple $(X,Y,H)$ satisfying the standard $\mathfrak{sl}(2)$-relations 
\begin{align}
[X,Y]=H, \qquad [H,X]=2X, \qquad [H,Y]=-2Y.
\end{align}
The classification of nilpotent orbits is therefore equivalent to the classification of embeddings $\mathfrak{sl}(2) \to \mathfrak{g}$. We shall often omit the subscript $X$ indicating the base point and simply write $\mathcal{O}$ for a nilpotent orbit.

 A convenient way of characterising an orbit is through its \emph{weighted Dynkin diagram}. This is simply the Dynkin diagram of $\mathfrak{g}$ with nodes labelled by the eigenvalues of the corresponding simple roots with respect to the Cartan generator $H$; a famous result asserts that these numbers can only be in the set $\{0,1,2\}$ (see \cite{CollingwoodMcGovern}). Another way of classifying the orbits is via their \emph{Bala-Carter labels}, which is the Dynkin-type of the unique minimal Levi subgroup $L_{\text{min}} \subset G$ which contains  $X$ as a \emph{distinguished} nilpotent element~\cite{BalaCarterI,BalaCarterII}. Here, `distinguished' refers to a grading of $\lie{l}=\mathrm{Lie}(L)=\oplus_i \lie{l}(i)$ by $H$ such that $\dim\lie{l}(0)=\dim\lie{l}(2)$.

Each Jacobson-Morozov triple $(X,Y,H)$ gives rise to a grading of the Lie algebra: 
\begin{align}
\mathfrak{g}=\mathfrak{g}(0)\oplus \bigoplus_{i=1}^m (\mathfrak{g}(i)\oplus \mathfrak{g}(-i))
\label{JMgrading}
\end{align} 
where $m$ is a number that depends on $\mathcal{O}_X$ through its weighted Dynkin diagram, and the graded pieces are defined as 
\begin{align} 
\mathfrak{g}(i)=\{x\in \mathfrak{g}\, |\, [H,x]=ix\}.
\end{align}
 The non-negative part gives rise to a parabolic subgroup $P_\mathcal{O}$ with unipotent radical $U_\mathcal{O}$ corresponding to the positive $H$-eigenvalues: $U_\mathcal{O}=\prod_{i=1}^{m}U_i$. The parabolic $P_\mathcal{O}$ can be read off from the weighted Dynkin diagram as follows. The Levi part $L_\mathcal{O}$ is generated by the maximal torus of $G$ and all copies of $SL(2)$ corresponding to those nodes in the weighted Dynkin diagram which are labelled by zeroes. The dimension of a given orbit is then determined by the formula $\dim \Oh = \dim \mathfrak{g} - \dim \mathfrak{g}(0) - \dim \mathfrak{g}(1)$.

\begin{example}
\label{E7minorb}
As an example we consider the minimal orbit of $G=E_7$. In Bourbaki labelling the weighted Dynkin diagram is $[1,0,0,0,0,0,0]$ and the semi-simple part of the Levi subgroup $L_\Oh$ is $D_6$. The graded pieces (of non-negative degree) are
\begin{align}
\mathfrak{g}(0) &= \mathfrak{so}(6,6) \oplus \mathfrak{gl}(1),\nn\\
\mathfrak{g}(1) &={\bf 32},\nn\\
\mathfrak{g}(2) &= {\bf 1},
\end{align}
where we have labelled positive degree pieces by their $\mathfrak{so}(6,6)$ representations. The dimension of the associated nilpotent orbit is $\dim\Oh_{\textrm{min}} = 133-67-32 = 34$. The nilpotent element of the Jacobson--Morozov triple can be chosen to be $X=E_{\alpha_1}$, the step operator of the first simple root. The Bala--Carter label of this orbit is $A_1$ (as for all minimal orbits of simply-laced groups) and the minimal Levi subgroup $L_{\textrm{min}}$ with semi-simple part $SL(2)$ is just the one associated with the first simple root. Clearly, $X=E_{\alpha_1}$ is a distinguished element in this group. 
\end{example} 
 
Associated with $(P_\mathcal{O}, L_\mathcal{O}, U_\mathcal{O})$ we also have the corresponding  Lie algebras $(\mathfrak{p}_\mathcal{O}, \mathfrak{l}_\mathcal{O}, \mathfrak{u}_\mathcal{O})$  that can be read off from the grading of $\mathfrak{g}$:
\begin{align} 
\mathfrak{p}_\mathcal{O}= \sum_{i\geq 0}\mathfrak{g}(i), \quad \mathfrak{u}_\mathcal{O}=\sum_{i>0}\mathfrak{g}(i), \qquad \mathfrak{l}_\mathcal{O}=\mathfrak{g}(0).
\end{align}

The parabolic subgroup $P_\mathcal{O}$ is uniquely determined by the nilpotent element $X$ which defines the orbit $\mathcal{O}$. 

To each orbit we also associate a stabilizer type in the following manner. Let $X$ be an element of a nilpotent orbit $\Oh$ and part of a Jacobson-Morozov triple $(X, Y, H)$ associated to that orbit such that $H \in \lie h$. Let $C_G(X) = \{g \in G : gXg^{-1} = X\}$ be its centralizer. Then the connected component $C_G(X)^0$ can be factorised into $C_G(X)^0 = R C$ where $R$ is the unipotent radical of $C_G(X)^0$ and $C = C_G(X)^0 \cap C_G(H)$ is a connected reductive group \cite{Carter}. 
While the reductive group $C$ depends on the representative $X$ in $\Oh$, the type of $C$ does not and it is this type that we associate to the orbit $\Oh$.

Finally, there exists a partial order on the space of all orbits defined by $\mathcal{O}\leq \mathcal{O}'$ if and only if $\overline{\mathcal{O}} \subset \overline{\mathcal{O}'}$ where the closure is defined using the Zariski topology. The partial order allows to arrange the nilpotent orbits on a Hasse diagram~\cite{CollingwoodMcGovern}.

\subsection{Degenerate Whittaker vectors and Fourier coefficients}
\label{sec:deg-Whittaker}

We shall now apply the formalism of nilpotent orbits to Fourier coefficients of Eisenstein series and relate it to degenerate Whittaker vectors. 

Let $P$ be a parabolic subgroup with unipotent radical $U$ and Levi subgroup $L$. As was recalled in~\eqref{eq:L-conjugation} the Fourier coefficients $F_{\psi_U}$ of an automorphic form $\varphi$ form orbits under $L(\rats)$, as do the characters $\psi_U$ on $U(\mathbb{A})$ trivial on $U(\mathbb{Q})$. Similar to~\eqref{Ncharacter}, we can describe the characters $\psi_U$ in terms of elements $\omega\in\mathfrak{u}^*$ such that for $u=\exp(X)$ for some $X\in \mathfrak{u}$ one has 
\begin{align}
\psi_U(u) = \e (\omega ( X ) ).
\end{align}
The element $\omega$ labels the charges dual to the coordinates on $U$. Using the Killing form one can equally regard it as an element of the Lie algebra $\mathfrak{u}$ and it is a nilpotent element. The Fourier coefficients $F_{\psi_U}$ are therefore related along $L(\rats)$-orbits of nilpotent elements in the radical. It is at this stage convenient to use a slightly coarser description and complexify all groups and Lie algebras.

Since $L\subset G$, the $L$-orbit of a character $\psi_U$ is contained in some nilpotent $G$-orbit. In order to determine whether a certain Fourier coefficient $F_{\psi_U}$ is supported by the wavefront set of an automorphic function $\varphi$, it is therefore important to know which $L$-orbits on $\mathfrak{u}$ exist and which nilpotent $G$-orbits they are contained in. This information has been computed recently by Miller and Sahi for all simple groups in the complexified setting~\cite{MillerSahi}.

\begin{example}
\label{SL4NTM}
In anticipation of our analysis below, we illustrate this for $G=SL(4)$. If we choose the maximal parabolic $P$ associated to the middle node of the $A_3$ Dynkin diagram, we have that $P=LU$ with $\mathfrak{l}= \mathfrak{sl}(2)\oplus\mathfrak{sl}(2)\oplus \mathfrak{gl}(1)$ and four-dimensional $\mathfrak{u}= ({\bf 2},{\bf 2})$ that can be thought of as the upper right $(2\times 2)$-block of a fundamental $SL(4)$-matrix. There are two non-trivial $L$-orbits on $\mathfrak{u}$ that can be represented by the $(2\times 2)$ matrices
\begin{align}
\begin{pmatrix}0&0\\1&0\end{pmatrix}\quad\textrm{and}\quad
\begin{pmatrix}0&1\\1&0\end{pmatrix}.
\end{align}
The first one lies in the (minimal) nilpotent $G$-orbit with Bala--Carter label $A_1$ (since its embedding in $\mathfrak{g}$ is the simple Chevalley generator of the middle node) and the second matrix lies in the (next-to-minimal) nilpotent $G$-orbit with label $2A_1$.
\end{example}

We will be particularly interested in the case of maximal parabolic subgroups $P$ of $G$, that is, parabolic subgroups such that the associated nilpotent subalgebra $\lie u \subset \lie g$ only contains a single simple root $\alpha$.  
As we know that $\mathfrak{u}$ contains a single simple root $\alpha$, we can consider the character $\psi_U$ which has support only on this root. As explained in \cite{MillerSahi}, such a representative is readily obtained by restricting a generic character $\psi$ of $N$ to the unipotent $U$. Since $\psi$ is determined by its values on the simple roots $\alpha \in \Pi$, and $U$ only contains a single simple root $\alpha$, we can view $\psi_{U}$ as the restriction of a (maximally) degenerate character $\psi_{\alpha}:N\to U(1)$ to $U \subset N$, according to
\begin{equation}
    \psi_U = \psi_\alpha\big|_U 
    \label{characterrestrict}
\end{equation}
and is therefore effectively a periodic function $\rats \bs \ads \to U(1)$ that only depends on a single variable. From the point of view of (\ref{Ncharacter}) this corresponds to a $\psi=\psi_\alpha$ such that when evaluated on an arbitrary $n\in N$ all coefficients $m_{\alpha'}=0$ except when $\alpha'=\alpha$, the simple root defining the maximal parabolic. 

To such a degenerate character $\psi_\alpha$ one associates naturally the \emph{degenerate} Whittaker vector (cf.~\eqref{Wvector})
\begin{equation}
    \label{Wdeg}
    W_{\psi_\alpha}(\chi, g) = \intl[N(\rats)\backslash N(\ads)] E(\chi, ng) \overline{\psi_\alpha(n)} \, dn
\end{equation}
which, as all Whittaker vectors, is determined by its values on $A\subset L$, cf.~\eqref{Wfunc}. 

It is now natural to wonder whether in this situation there is any relation between the coefficient 
\begin{equation}
   F_{\psi_U}(\chi, g) = \intl[U(\rats) \bs U(\ads)] E(\chi, ug) \overline{\psi_U(u)} \, du
\end{equation}
and the degenerate Whittaker vector $W_{\psi_\alpha}$ in \eqref{Wdeg}. Clearly, the transformation law is quite different in general for the two functions; in particular, $F_{\psi_U}$ is determined by a function on $L(\mathbb{A})$ while  $W_{\psi_\alpha}$ is determined by a function only on the Cartan torus $A(\mathbb{A})\subset L(\mathbb{A})$. When $E(\chi, g)$ is in the minimal representation the support of $F_{\psi_U}$ is reduced and, remarkably, it turns out that under certain circumstances the two functions $F_{\psi_U}$ and $W_{\psi_\alpha}$ agree. An example for $SL(3)$ can be seen in \eqref{eq:SL3-Fourier-Whitt-agree}.

The first piece of information we need comes from the main theorem of Miller-Sahi \cite{MillerSahi}. The theorem states that for $G(\ads)$ being the adelization of either the exceptional group $E_6$ or $E_7$, an automorphic form $\varphi$ in the minimal representation $\pi_\text{min}$ is completely determined by the maximally degenerate Whittaker vectors on $N$. A version of this theorem for $G(\ads) = SL(3, \ads)$ and $G(\ads) = SL(4, \ads)$ has already been mentioned in the introduction and will be proven in sections \ref{sec:SL3} and \ref{sec:SL4}. 

In the proof of this theorem, Miller-Sahi start with a certain maximal parabolic subgroup $P$ of $G$ with abelian radical $U$ and express $\varphi$ as a sum of Fourier coefficients $F_{\psi_U}$. Using arguments from Matumoto and Mœglin-Waldspurger they conclude that in the minimal representation, only Fourier coefficients with $\psi$ being either trivial or $L(\rats)$-equivalent to $\psi|_U$ in \eqref{characterrestrict} are non-vanishing. By iteratively doing further expansions along the stabilizer of $\psi_U$ and applying Matumoto and Mœglin-Waldspurger they conclude that $F_{\psi_U}$ is an $L$-translate of the degenerate Whittaker vector with character $\psi_\alpha$ as defined in \eqref{characterrestrict}.

Our purpose is to generalise this process of expressing Fourier coefficients in terms of $L$-translates of Whittaker vectors that are easier to evaluate. At the same time we obtain a powerful method for determining whether certain Fourier coefficients of $\varphi$ vanish and when the two functions $F_{\psi_U}$ defined in \eqref{Fcoefficient} and $W_{\psi_\alpha}$ from \eqref{Wdeg} are actually equal. Our strategy is to combine the results of Miller-Sahi with the techniques developed by Ginzburg \cite{MR2214128} for attaching Fourier coefficient to nilpotent orbits. We shall also give evidence for a generalisation of Miller-Sahi's theorem to next-to-minimal automorphic representations where we consider characters $\psi$ supported on two orthogonal simple roots (type $2A_1$).

\subsection{Fourier coefficients, nilpotent orbits and small representations}
\label{sec:orbit-rep}

Ginzburg  has developed a method for canonically associating Fourier coefficients of automorphic forms with nilpotent orbits \cite{MR2214128} (see also \cite{MR3115211,MR3161096}). This gives a practical way of structuring the Fourier expansion, which is particularly useful for determining which coefficients vanish because of representation theoretic constraints. In this section we shall first outline this procedure, and then we shall apply it to the specific questions of interest in this paper. 

Following \cite{MR2214128} we denote by $U_{i\geq n}$ the subspace corresponding to all $H$-eigenvalues $i\geq n$. Denote by $V$ the subspace $U_{i\geq 2}$ which contains the Jacobson-Morozov representative $X$ of the nilpotent orbit $\mathcal{O}_X$, and let $\psi_V$ be a non-trivial character on $V/[V,V]$, trivial on $V/[V,V](\mathbb{Q})$ and fixed by a stabiliser $C(\psi_V)$ in $L_\Oh(\rats)$ of the same type as that of the stabilizer associated to the orbit $\mathcal{O}$. This last condition is quite important and limits the choice of character $\psi_V$ strongly. 

For example, when $\mathcal{O}$ is the regular orbit, then $P_\mathcal{O}=B$, the Borel subgroup, and the Levi decomposition $P_\mathcal{O}=L_\mathcal{O}U_\mathcal{O}$ is nothing but the Iwasawa decomposition $B=AN$, where $A$ is the Cartan torus and $N$ is the maximal unipotent radical generated by all the positive roots. In this case $V=N$ and hence $\psi_V$ is a  character on all of $N$. Moreover, the stabiliser $C$ of the regular orbit is trivial, and therefore the character $\psi_V$ must be a generic character on $N$, since these are the only characters with at most a trivial stabiliser. 

We shall now attach yet another object to each nilpotent orbit, namely a Fourier coefficient of an automorphic form. Let $\varphi$ be a vector in an automorphic representation $\pi$ of a reductive group $G(\mathbb{A})$. We then have: 

\vspace{.3cm}

\begin{definition}
    To each choice of data $(\pi, \varphi, \mathcal{O}, \psi^{\mathcal{O}}_V)$ we can associate a Fourier coefficient as follows:
    \begin{equation}
        F_\Oh(\varphi, \psi^{\mathcal{O}}_V; g) := \int_{V(\mathbb{Q})\backslash V(\mathbb{A})} \varphi(vg)\overline{\psi^{\mathcal{O}}_V(v)}dv.
        \label{OrbitFC}
    \end{equation}
\end{definition}

\vspace{.3cm} 

\begin{remark}
Note that (\ref{OrbitFC}) defines a function on $G(\mathbb{A})$ which is left-invariant under $C(\mathbb{Q})$. Hence, by restriction one can view it as a function on $C(\mathbb{Q})\backslash C(\mathbb{A})$. This does not, however, define an automorphic form on $C(\mathbb{A})$ since in general (\ref{OrbitFC}) fails to be $\mathcal{Z}(\mathfrak{c})$-finite \cite{MR3115211}.\footnote{$\mathcal{Z}(\mathfrak{c})$ denotes the centre of the universal enveloping algebra $\mathcal{U}(\mathfrak{c})$ of the Lie algebra $\mathfrak{c}$ of $C$.}
\end{remark}

\vspace{.3cm} 

We say that  a Fourier coefficient defined by (\ref{OrbitFC}) is attached to the orbit $\Oh$ and will for short call it an orbit coefficient. As a special case we have for the trivial orbit that
\begin{align}
F_{\Oh_0} (\varphi;g) = \int_{N(\rats)\bs N(\ads)} \varphi(ng) dn
\end{align}
yields the constant term of $\varphi$. For Eisenstein series $E(\chi,g)$, we will also denote the constant term by $E_0(\chi,g)$.

The virtue of the above definition of an orbit coefficient is that it allows to use representation theory to assert whether certain Fourier coefficients vanish or not. 
For example, in the case of a minimal representation $\pi_\text{min}$ it is known by theorems of Matumoto \cite{MR892192} (at the real place $p=\infty$) and Moeglin-Waldspurger \cite{MR913667} (at the finite places $p<\infty$) that the only coefficients which are non-zero are those attached to the trivial orbit $\mathcal{O}_0$ and the smallest non-trivial orbit $\mathcal{O}_\text{min}$. See the Remark in the introduction for a discussion of these results. As we shall see momentarily, the formalism above allows to immediately write down these non-vanishing coefficients.

\section{\texorpdfstring{$SL(3)$}{SL(3)} Eisenstein series}
\label{sec:SL3}

We shall now illustrate the formalism in the special case of Eisenstein series on $SL(3, \mathbb{A})$. To this end let $\chi : B(\mathbb{Q})\backslash B(\mathbb{A})\to \mathbb{C}^{\star}$ be a quasi-character defined by $\chi(b)=\chi(na)=\chi(a)$, and extended to all of $G$ by right-invariance under $K$. The function on $G$ so obtained is the unique (up to scaling) spherical vector of the non-unitary principal series $\text{Ind}_{B(\mathbb{A})}^{G(\mathbb{A})}\chi$. With a slight abuse of notation we shall denote the spherical vector by the same letter $\chi$ as for the inducing character. Let $\pi^\chi$ be an automorphic realisation of this representation. The image of the spherical vector $\chi$ inside $\pi^\chi$ is given by the Eisenstein series 
\begin{equation}
    E(\chi, g)=\sum_{\gamma\in B(\mathbb{Q})\backslash G(\mathbb{Q})} \chi(\gamma g).
\end{equation}

On general grounds the total Fourier expansion of $E(\chi, g)$ takes the form
\begin{equation}
    E(\chi, g)= E_0(\chi, g)+\sum_{\psi_N\neq 1} W_{\psi_N}(\chi, g) + \sum_{\psi_{\mathcal{Z}}\neq 1} W_{\psi_\mathcal{Z}} (\chi, g),
    \label{GeneralFE}
\end{equation}
where $\psi_N$ is a unitary character on $N(\ads)$, trivial on $N(\rats)$, and $\psi_{\mathcal{Z}}$ is a generic character on the centre $\mathcal{Z}(\ads)=[N(\ads), N(\ads)]$, trivial on $\mathcal{Z}(\mathbb{Q})$.\footnote{This is the Heisenberg grading of $SL(3)$.} Note that $\psi_N$ is necessarily also trivial on $\mathcal{Z}$ and so restricts to a character on the abelianization $\mathcal{Z}\backslash N$. The (non-constant) coefficients are defined by the integrals
\begin{align}
W_{\psi_N}(\chi, g) &= \int_{N(\mathbb{Q})\backslash N(\mathbb{A})} E(\chi, ng)\overline{\psi(n)}dn, 
\nn \\ 
W_{\psi_\mathcal{Z}}(\chi, g) &= \int_{\mathcal{Z}(\mathbb{Q})\backslash \mathcal{Z}(\mathbb{A})} E(\chi, zg)\overline{\psi_{\mathcal{Z}}(z)}dz.
\end{align}

We note that the first two terms in \eqref{GeneralFE} sum to the remaining Whittaker vector on $\mathcal Z$ with trivial character $\psi_\mathcal{Z}=1$. We call $W_{\psi_N}$ an \emph{abelian} coefficient and $W_{\psi_Z}$ a \emph{non-abelian}  coefficient. Explicit formulas for the coefficients when $g=(g_\mathbb{R}; 1, 1, \dots, 1), \, g_\mathbb{R}\in SL(3,\mathbb{R}),$ can be extracted from Bump's thesis \cite{Bump} or Vinogradov-Takhtajan's paper \cite{VT}. 

For the generic principal series with Gelfand-Kirillov dimension
\beq
\GKdim(\pi^\chi) = \frac{\dim G -\rank G}{2} = \dim N = 3,
\eeq
all these abelian and non-abelian Fourier coefficients are non-vanishing. 

Using the duality between characters $\chi$ and complex weights $\lambda \in \mathfrak{h}^{\star}\otimes \mathbb{C}\cong \mathbb{C}^2$ ($\mathfrak{h}=\text{Lie}\, A$) we can parametrize $\pi^\chi$ by two complex numbers $(s_1, s_2)\in \mathbb{C}^2$ (in a suitable basis of fundamental weights). For certain special loci in the parameter space the Gelfand-Kirillov dimension of $\pi^\chi=\pi^{s_1, s_2}$ reduces. In particular, we can consider the sub-representation $\pi_{\text{min}}=\pi^{s, 0}$ which is the minimal representation (or degenerate principal series) of Gelfand-Kirillov dimension
\beq
\GKdim(\pi_{\text{min}})=2.
\eeq
In this situation it turns out that the Fourier coefficients simplify~\cite{Pioline:2009qt}. 

Generally, we can parametrize the character $\psi_N$ explicitly by:
\begin{equation}
    \psi_N(n)=\psi_N\left(e^{xE_{\alpha_1}+yE_{\alpha_2}}\right)=\e(px+qy)\equiv e^{2\pi i (px+qy)}, \quad (p,q)\in \mathbb{Q}^2, \quad (x,y)\in \mathbb{A}^2,
\end{equation}
where $E_{\alpha_1}$ and $E_{\alpha_2}$ are the standard Chevalley generators of the Lie algebra $\mathfrak{sl}(3)$ associated with the simple roots $\alpha_1, \alpha_2$. The character does not depend on the one-parameter subgroup associated with the third (non-simple) positive root $\alpha_1+\alpha_2$ since $x_{\alpha_1+\alpha_2}(u)\in [N,N]$ and $\psi_N$ is therefore trivial on such elements. To emphasize the dependence on $p,q$ we sometimes denote the character $\psi_N$ by $\psi_{p,q}$. When both $p$ and $q$ are non-zero we call the character $\psi_{p,q}$ \emph{generic} and when either $p=0, q\neq0$ or $p\neq 0, q=0$ we call $\psi_{0,q}$ (or $\psi_{p,0}$) \emph{degenerate}. 

For the minimal representation $\pi_\text{min} = \pi^{s,0}$ one then finds that all the abelian coefficients associated with generic characters vanish:
\begin{equation}
W_{\psi_{p,q}}(s, 0; g)=0, \qquad \psi_{p,q} \, \, \text{generic} \quad (p\neq 0\,\text{and}\, q\neq 0). 
\label{minconstraint}
\end{equation}
In other words, the abelian Fourier coefficients are completely determined by the degenerate coefficients  $W_{\psi_{0,q}}(s, 0; g)$ and $W_{\psi_{p,0}}(s, 0; g)$. For the non-abelian coefficients there are no vanishing properties but the explicit form of $W_{\psi_Z}$ simplifies at each local place. See for instance sec. 3.4 and App. B of \cite{Pioline:2009qt} for an analysis of this.

\subsection{Fourier coefficients attached to nilpotent orbits}
We would now like to analyse the same Fourier expansion from the point of view of nilpotent orbits as outlined in the previous section. For $SL(3)$ there are three orbits: the principal (or regular) nilpotent orbit $\mathcal{O}_{\text{reg}}$ of dimension $6$ and Bala--Carter label $A_2$, the minimal orbit (or sub-regular orbit) $\mathcal{O}_{\text{min}}$ of dimension 4 and Bala--Carter label $A_1$, and the trivial orbit $\mathcal{O}_0$ of dimension 0 and label $0$. 

For the regular orbit $\mathcal{O}_{\text{reg}}$ we have $V=U_{i\geq1}=U_{i\geq2}\equiv N$, the unipotent radical of the Borel $B$, and hence the character $\psi^{\mathcal{O}_{\text{reg}}}_V$ coincides with the generic character $\psi_N : N(\mathbb{Q})\backslash N(\mathbb{A})\to U(1)$ introduced in the previous section. Therefore, in this case the coefficient (\ref{OrbitFC}) becomes 
\begin{equation}
    F_{\Oh_\text{reg}}(\chi, \psi^{\Oh_\text{reg}}_V; g) = \intl[N(\mathbb{Q})\backslash N(\mathbb{A})] E(\chi, ng)\overline{\psi_N(n)} dn, 
    \label{regularFC}
\end{equation}
and is identified with the generic part (i.e. $p q\neq 0$) of the abelian Fourier coefficient $W_\psi$ in \eqref{GeneralFE} since the stabilizer is trivial. 

For the minimal orbit $\mathcal{O}_{\text{min}}$ we have instead $V=U_{i\geq2}=\mathcal{Z}$ (and $U_{i\geq1}\neq U_{i\geq2}$), the center $[N,N]$  of the Heisenberg group $N$, and therefore the character $\psi_V^{\mathcal{O}_{\text{min}}}$ is simply the generic character $\psi_\mathcal{Z} : \mathcal{Z}(\mathbb{Q})\backslash \mathcal{Z}(\mathbb{A})\to U(1)$, and the Fourier coefficient (\ref{OrbitFC}) attached to the minimal orbit becomes
\begin{equation}
    F_{\Oh_\text{min}}(\chi, \psi_V^{\Oh_\text{min}}; g)=\intl[\mathcal{Z}(\mathbb{Q})\backslash \mathcal{Z}(\mathbb{A})] E(\chi, zg)\overline{\psi_\mathcal{Z}(z)}dz.
    \label{minFC}
\end{equation}
Hence, this is nothing but the non-abelian coefficient $W_{\psi_\mathcal{Z}}$ in (\ref{GeneralFE}). 

We are however still missing the degenerate abelian Fourier coefficients associated with characters $\psi_{p, 0}$ and $\psi_{0, q}$ on $N$. As we shall see below in~\eqref{OtherMin}, these are in fact 
also connected to the minimal orbit $\mathcal{O}_\text{min}$.  This is due to the fact that $U_{i\geq1}\neq U_{i\geq 2}$ for the minimal orbit and we must therefore take into account that one can have non-trivial characters along $U_{i\geq1}/U_{i\geq2}$. 

We will parametrize the element $z \in \mathcal{Z}$ with a parameter $x \in \ads$ as shown below and instead of considering a character $\psi_V$ on $V(\rats) \bs V(\ads)$ we will work with the (non-trivial) character $\e : \rats \bs \ads \to U(1)$ with $\e(x)=e^{2\pi i x}$ defined in~\eqref{eNotation}.

In this notation the Fourier coefficient attached to the minimal orbit (that is, the non-abelian coefficient) can be expressed as
\begin{equation}
    \label{eq:SL3-min-orbit-coeff}
    F_{\Oh_\text{min}}(\chi, m'; g) = 
    \intl_{\rats \bs \ads} E\Big(\chi,
    \begin{psmallmatrix}
        1 & & x \\
        & 1 & \\
        & & 1
    \end{psmallmatrix}
    g\Big) \overline{\e(m' x)} \, dx
\end{equation}
where the charge $m' \in \rats^*$ characterizes the non-trivial $\psi_V^{\Oh_\text{min}}$ \footnote{Henceforth, charges decorated with a prime will be assumed to be non-zero.}.

We note that, using the automorphic invariance of $E(\chi, g)$
\begin{equation}
\begin{split}
    F_{\Oh_\text{min}}\left(\chi, m';
    \begin{psmallmatrix}
        -1 & & \\
        & & -1 \\
        & -1 &
    \end{psmallmatrix} g \right) 
    &=
    \intl_{\rats \bs \ads} E\Big(\chi,
    \begin{psmallmatrix}
        1 & & x \\
        & 1 & \\
        & & 1
    \end{psmallmatrix}
    \begin{psmallmatrix}
        -1 & & \\
        & & -1 \\
        & -1 &
    \end{psmallmatrix} 
    g\Big) \overline{\e(m' x)} \, dx \\
    &=
    \intl_{\rats \bs \ads} E\Big(\chi,
    \begin{psmallmatrix}
        -1 & & \\
        & & -1 \\
        & -1 &
    \end{psmallmatrix}^{-1}
    \begin{psmallmatrix}
        1 & & x \\
        & 1 & \\
        & & 1
    \end{psmallmatrix}
    \begin{psmallmatrix}
        -1 & & \\
        & & -1 \\
        & -1 &
    \end{psmallmatrix} 
    g\Big) \overline{\e(m' x)} \, dx \\
    &=
    \intl_{\rats \bs \ads} E\Big(\chi,
    \begin{psmallmatrix}
        1 & x & \\
        & 1 & \\
        & & 1
    \end{psmallmatrix}
    g\Big) \overline{\e(m' x)} \, dx
\end{split}
\end{equation}

This means that we can connect the minimal orbit coefficient with the degenerate Whittaker vector $W_{\psi_{m',0}}$ by
\begin{multline}
\label{OtherMin}
    \intl_{(\rats \bs \ads)^2} 
    F_{\Oh_\text{min}}\left(\chi, m';
    \begin{psmallmatrix}
        -1 & & \\
        & & -1 \\
        & -1 &
    \end{psmallmatrix}
    \begin{psmallmatrix}
        1 & & u_2 \\
        & 1 & u_3 \\
        & & 1
    \end{psmallmatrix} g \right) \, d^2u = \\
    = \intl_{(\rats \bs \ads)^3} E\Big(\chi,
    \begin{psmallmatrix}
        1 & x & u_2' \\
        & 1 & u_3 \\
        & & 1
    \end{psmallmatrix}
    g\Big) \overline{\e(m' x)} \, dx \, d^2u
    = W_{\psi_{m', 0}}(\chi, g)
\end{multline}
where we have made the substitution $u_2 + u_3 x \to u_2'$.

We note that the vanishing of $F_{\Oh_\text{min}}$ implies the vanishing of the degenerate Whittaker vector $W_{\psi_{m',0}}$. When proving theorem \ref{thm:min-rep} for $SL(3)$ we will, for the minimal representation, be able to prove the converse by expressing $F_{\Oh_\text{min}}$ as a sum of Whittaker vectors $W_{\psi_{m', 0}}$.

\subsection{Expansion in nilpotent orbits and Theorem \ref{thm:orbit-exp} for \texorpdfstring{$SL(3)$}{SL(3)}}
\label{sec:SL3-orbit-expansion}
Similar to the way the Whittaker vector $W_{\psi_{m',0}}$ was related to $F_{\Oh_\text{min}}$ above, one can express all the Whittaker vectors in \eqref{GeneralFE} in terms of orbit coefficients.
\begin{equation}
\begin{split}
    \label{eq:SL3-Whittaker-as-orbits}
    W_{\psi_{m', n'}}(\chi, g) &= F_{\Oh_\text{reg}}(\chi, m', n'; g) \\
    W_{\psi_{m', 0}}(\chi, g) &= \intl_{(\rats \bs \ads)^2} 
    F_{\Oh_\text{min}}\left(\chi, m';
    \begin{psmallmatrix}
        -1 & & \\
        & & -1 \\
        & -1 &
    \end{psmallmatrix}
    \begin{psmallmatrix}
        1 & & u_2 \\
        & 1 & u_3 \\
        & & 1
    \end{psmallmatrix} g \right) \, d^2u \\
    W_{\psi_{0, n'}}(\chi, g) &= \intl_{(\rats \bs \ads)^2} 
    F_{\Oh_\text{min}}\left(\chi, n';
    \begin{psmallmatrix}
        & -1 & \\
        -1 & &  \\
        & & -1
    \end{psmallmatrix}
    \begin{psmallmatrix}
        1 & u_1 & u_2 \\
        & 1 & \\
        & & 1
    \end{psmallmatrix} g \right) \, d^2u \\
    W_{\psi_{k'}}(\chi, g) &= F_{\Oh_\text{min}}(\chi, k'; g)
\end{split}
\end{equation}
where the last Whittaker vector is over $\mathcal{Z}$ and $k' \in \rats^*$ characterizes $\psi_{\mathcal{Z}}$. The constant term $E_0(\chi, g)$ is attached to the trivial orbit. Recall that primed charges are assumed to be non-zero. From now on we will suppress the integration domain for brevity.  

\begin{proof}[\textbf{Proof of Theorem \ref{thm:orbit-exp} for $SL(3)$}]~\par
    Together with \eqref{GeneralFE}, the rewriting of the Whittaker vectors in \eqref{eq:SL3-Whittaker-as-orbits} proves Theorem \ref{thm:orbit-exp} for $SL(3)$; we can expand $E(\chi, g)$ in terms of orbit coefficients.

By regrouping the summations in \eqref{GeneralFE} we get that
\begin{equation}
    E(\chi, g) = \mathcal F_{\Oh_0}(\chi, g) + \mathcal F_{\Oh_\text{min}}(\chi, g) + \mathcal F_{\Oh_\text{reg}}(\chi, g)
\end{equation}
with
\begin{equation}
    \begin{split}
        \mathcal F_{\Oh_0}(\chi, g) &= E_0(\chi, g) \\
        \mathcal F_{\Oh_\text{min}}(\chi, g) &= \sum_{m' \neq 0} \left(W_{\psi_{m', 0}}(\chi, g) + W_{\psi_{0,m'}}(\chi, g) \right) + \sum_{k' \neq 0} W_{\psi_{k'}}(\chi, g) \\
        \mathcal F_{\Oh_\text{reg}}(\chi, g) &= \suml[m',n' \neq 0] W_{\psi_{m',n'}}(\chi, g) 
    \end{split}
\end{equation}
where each $\mathcal F_\Oh$ is linearly determined by $F_\Oh$ coefficients as seen in \eqref{eq:SL3-Whittaker-as-orbits}.
\end{proof}

\subsection{Example for Fourier coefficient on maximal parabolic subgroup}

Other Fourier coefficients on general parabolic subgroups can be connected to the orbit coefficients in a similar way. Let us consider, for example, the (maximal) parabolic subgroup $P$ with unipotent subgroup
\begin{equation}
    U = \left\{ 
    \begin{psmallmatrix}
        1 & u_1 & u_2 \\
        & 1 & \\
        & & 1
    \end{psmallmatrix} : u_i \in \ads \right\}
\end{equation}
and the Fourier coefficient
\begin{equation}
    \label{eq:SL3-parabolic-coeff}
    F_U(\chi, m_1, m_2; g) = \int E\left(\chi,
    \begin{psmallmatrix}
        1 & u_1 & u_2 \\
        & 1 & \\
        & & 1
    \end{psmallmatrix} g\right) \overline{\e(m_1 u_1 + m_2 u_2)} \, d^2u 
\end{equation}
where we assume that $m_1, m_2 \in \rats$ not both zero.

As discussed at the end of section \ref{sec:fourier-coeffs} we can use conjugation by elements in $L(\rats)$ to transform the character into a simpler form
\begin{equation} 
    \label{eq:SL3-FU}
    F_U(\chi, m_1, m_2; g) = F_U(\chi, 0, m'; l_U g) = \int E\left(\chi,
    \begin{psmallmatrix}
        1 & u_1 & u_2 \\
        & 1 & \\
        & & 1
    \end{psmallmatrix} l_U g\right) \overline{\e(m' u_2)} \, d^2u 
\end{equation}
where
\begin{equation}
    \label{eq:SL-FU-l}
    \left\{
        \begin{aligned}
            l_U &=
            \begin{psmallmatrix}
                1 & & \\
                & 1 & -\frac{m_1}{m_2}\\
                & & 1
            \end{psmallmatrix}
            & m' = m_2
            & \qquad \text{ if } m_2 \neq 0 \\
            l_U &= 
            \begin{psmallmatrix}
                -1 & & \\
                & & -1 \\
                & -1 &
            \end{psmallmatrix}
            & m' = m_1
            & \qquad \text{ if } m_2 = 0
        \end{aligned}
    \right.
\end{equation}

The first case is seen by conjugating the prefactor of $g$ with $l_U$ and making the variable change $u_2 + \frac{m_1}{m_2}u_1 \to u_2$ while the second case simply switches $u_1$ and $u_2$.\footnote{See \eqref{eq:SL3-Fmin-expansion-step} for a similar, more detailed derivation.}

Now we can connect this Fourier coefficient to the minimal orbit coefficient by
\begin{equation}
    \label{eq:SL3-FU-canonical}
    F_U(\chi, 0, m'; l_U g) = \int F_{\Oh_\text{min}}\left( \chi, m';  
        \begin{psmallmatrix}
            1 & u_1 & \\
            & 1 & \\
            & & 1
        \end{psmallmatrix} l_U g \right) \, du_1 \, .
\end{equation}

In the next section we will see how we can relate this Fourier coefficient with Whittaker vectors on $N$.

\subsection{The minimal representation and Theorem \ref{thm:min-rep} for \texorpdfstring{$SL(3)$}{SL(3)} }

From section \ref{sec:orbit-rep} we know that when $E(\chi, g)$ is in the minimal representation, the only orbit coefficients that do not vanish are those attached to $\Oh_\text{min}$ or $\Oh_\text{0}$ (the trivial orbit). This motivates us to take a closer look at $F_{\Oh_\text{min}}$, which also simplifies when in the minimal representation. We illustrate this by expanding $F_{\Oh_\text{min}}$ itself into different, conjugated, orbit coefficients. 

We Fourier expand the integrand in \eqref{eq:SL3-min-orbit-coeff} as
\begin{equation}
    \label{eq:SL3-min-integrand}
    \begin{split}
        \MoveEqLeft
        E \left( \chi,
            \begin{psmallmatrix}
                1 & & x_1 \\
                & 1 & \\
                & & 1
            \end{psmallmatrix}
            g \right) \overline{\e(m_1' x_1)}
        = \\
        =& \sum_{m_2 \in \rats} \, \int E \left( \chi,
            \begin{psmallmatrix}
                1 & x_2 & \\
                & 1 & \\
                & & 1
            \end{psmallmatrix}
            \begin{psmallmatrix}
                1 & & x_1 \\
                & 1 & \\
                & & 1
            \end{psmallmatrix}
            g \right) \overline{\e(m_1' x_1)} \, \overline{\e(m_2 x_2)} \, dx_2 \\
        =&
        \sum_{m_2 \in \rats} \, \int E \left( \chi, 
            \begin{psmallmatrix}
                1 & x_2 & x_1 \\
                & 1 & \\
                & & 1
            \end{psmallmatrix}
            g \right) \overline{\e(m_1' x_1 + m_2 x_2)} \, dx_2
    \end{split}
\end{equation}

Similar to the parabolic Fourier coefficient example above in \eqref{eq:SL3-parabolic-coeff} we can conjugate the prefactor of $g$ with
\begin{equation}
    l(m_1',m_2) = 
    \begin{psmallmatrix}
        -1 & & \\
        & & -1 \\
        & -1 & m_2/m_1'
    \end{psmallmatrix}
\end{equation}
and use the automorphic invariance of $E$ to get from \eqref{eq:SL3-min-integrand} inserted into \eqref{eq:SL3-min-orbit-coeff} that
\begin{equation}
    \label{eq:SL3-Fmin-expansion-step}
    \begin{split}
        \MoveEqLeft
        F_{\Oh_\text{min}}(\chi, m_1', g) = \\
        &=
        \sum_{m_2} 
        \int E\left(\chi,
        \begin{psmallmatrix}
            1 & x_2 & x_1 \\
            & 1 & \\
            & & 1
        \end{psmallmatrix}
        g\right) \overline{\e(m_1' x_1 + m_2 x_2)} \, dx_1 dx_2 \\
        &=
        \sum_{m_2} 
        \int E\left(\chi, l(m_1', m_2)
        \begin{psmallmatrix}
            1 & x_2 & x_1 \\
            & 1 & \\
            & & 1
        \end{psmallmatrix} l(m_1', m_2)^{-1} l(m_1', m_2)
        g\right) \overline{\e(m_1' x_1 + m_2 x_2)} \, d^2x \\
        &=
        \sum_{m_2}
        \int E\left(\chi,
        \begin{psmallmatrix}
            1 & x_1 + m_2 x_2 / m_1' & x_2 \\
            & 1 & \\
            & & 1
        \end{psmallmatrix}
        l(m_1', m_2) g\right) \overline{\e(m_1' x_1 + m_2 x_2)} \, d^2x \\
        &=
        \sum_{m_2}
        \int E\left(\chi,
        \begin{psmallmatrix}
            1 & x_1 & x_2 \\
            & 1 & \\
            & & 1
        \end{psmallmatrix}
        l(m_1', m_2) g\right) \overline{\e(m_1' x_1)} \, d^2x
    \end{split} 
\end{equation}
where, in the last step, the substitution $x_1 + m_2 x_2 / m_1' \to x_1$ has been made.

We make a further Fourier expansion to obtain
\begin{equation}
    F_{\Oh_\text{min}}(\chi, m_1'; g) = 
    \sum_{m_2,m_3}
    \int E\left(\chi,
    \begin{psmallmatrix}
        1 & x_1 & x_2 \\
        & 1 & x_3 \\
        & & 1
    \end{psmallmatrix}
    l(m_1', m_2) g\right) \overline{\e(m_1' x_1 + m_3 x_3)} d^3x
\end{equation}

We note that for the terms with $m_3 \neq 0$ we can identify $\e(m_1' x_1 + m_3 x_3)$ with a generic character on $N$. Thus, these terms are attached to $\Oh_\text{reg}$.
\begin{equation}
    \label{eq:SL3-orbit-sum}
    \begin{split}
        \MoveEqLeft
        F_{\Oh_\text{min}}(\chi, m_1'; g) = \\
        &= \sum_{m_3' \neq 0} \sum_{m_2} F_{\Oh_\text{reg}}(\chi, m_1', m_3'; l(m_1', m_2) g) 
        +{} \\ &\quad + 
        \sum_{m_2}
        \int E\left(\chi,
        \begin{psmallmatrix}
            1 & x_1 & x_2 \\
            & 1 & x_3 \\
            & & 1
        \end{psmallmatrix}
        l(m_1', m_2) g\right) \overline{\e(m_1' x_1)} d^3x \\
        &= \sum_{m_2, m_3'} F_{\Oh_\text{reg}}(\chi, m_1', m_3'; l(m_1', m_2) g) 
        +
        \sum_{m_2}
        W_{\psi_{m_1',0}}(\chi; l(m_1', m_2) g)
    \end{split}
\end{equation}

We note that when summing $l(m_1', m_2)$ over $m_2$ (keeping $m_1'$ fixed) we might as well take the sum of $l(n) = l(1, n)$ over $n \in \rats$.

\begin{proof}[\textbf{Proof of Theorem \ref{thm:min-rep} for $SL(3)$}]~\par
    If $E(\chi, g)$ is in the minimal representation (i.e. $\chi = \chi_\text{min}$), all orbit coefficients attached to nilpotent orbits outside the closure $\overline{\Oh_\text{min}}$ (that is, those attached to $\Oh_\text{reg}$) vanish using the arguments of section \ref{sec:orbit-rep}.
 
    The first sum in the expansion of $F_{\Oh_\text{min}}$ in \eqref{eq:SL3-orbit-sum} then vanishes leaving only the maximally degenerate Whittaker vectors. 
    \begin{equation}
        \label{eq:SL3-Fmin-min-rep}
        F_{\Oh_\text{min}}(\chi_\text{min}, m_1', g) = \sum_{n \in \rats} W_{\psi_{m_1',0}}(\chi_\text{min}; l(n) g) \qquad l(n) =
        \begin{psmallmatrix}
            -1 & & \\
            & & -1 \\
            & -1 & n 
        \end{psmallmatrix}
    \end{equation}
\end{proof}

These manipulations of the minimal orbit coefficient for $SL(3)$ are inspired by Ginzburg \cite{GinzburgUnpublished}.

Thus, both the non-abelian Whittaker vector $W_{\psi_\mathcal{Z}}$ and the parabolic Fourier coefficient $F_U$ in \eqref{eq:SL3-parabolic-coeff} depend only on these degenerate Whittaker vectors on $N$ in the minimal representation since they are directly given by $F_{\Oh_\text{min}}$. 

\vspace{.3cm} 

\begin{remark}
The fact that all Fourier coefficients of the minimal Eisenstein series on $SL(3)$ are captured by maximally degenerate Whittaker vectors is compatible with the 
earlier results in~\cite{Pioline:2009qt} (see section 3.4 there).
\end{remark}

\vspace{.3cm}

For the minimal representation, the full expansion of $E(\chi, g)$ in \eqref{GeneralFE} can be expressed in terms of $F_{\Oh_\text{min}}$ (and $F_{\Oh_0}$) and thus in maximally degenerate Whittaker vectors (and the constant term). The Whittaker vectors that are not connected to $F_{\Oh_\text{min}}$ vanish.

In this way we can understand an $SL(3)$ variant of Miller and Sahi's theorem explained in section \ref{sec:deg-Whittaker}: in the minimal representation, Fourier coefficients that cannot be connected to $F_{\Oh_\text{min}}$ vanish and those that are connected to $F_{\Oh_\text{min}}$ are expressed by maximally degenerate Whittaker vectors.

As discussed in the beginning of section \ref{sec:deg-Whittaker}, we will now show that, in the minimal representation, the Fourier coefficient $F_U(\chi, m_1, m_2; g)$ is an $L(\rats)$-translate of a Whittaker vector with character supported only on the root $\alpha_1$ corresponding to the parameter $u_1$ and that in some cases they completely agree.

From \eqref{eq:SL3-FU}--\eqref{eq:SL3-FU-canonical} and \eqref{eq:SL3-Fmin-min-rep}
\begin{equation}
    \begin{split}
        \MoveEqLeft
        F_U(\chi_\text{min}, m_1, m_2; g) = F_U(\chi_\text{min}, 0, m'; l_U g) \\
        &= \int F_{\Oh_\text{min}}\left(\chi_\text{min}, m';
            \begin{psmallmatrix}
                1 & u_1 & \\
                & 1 & \\
                & & 1
            \end{psmallmatrix}
            l_U g \right) \, du_1 \\ 
        &= \sum_{n} \int W_{\psi_{(m', 0)}} \left( \chi_\text{min};
            \begin{psmallmatrix}
                -1 & & \\
                & & -1 \\
                & -1 & n
            \end{psmallmatrix}
            \begin{psmallmatrix}
                1 & u_1 & \\
                & 1 & \\
                & & 1
            \end{psmallmatrix}
            l_U g \right) \, du_1 \\
        &= \sum_{n} \int E\left(\chi_\text{min},
            \begin{psmallmatrix}
                1 & x_1 + n u_1 & x_2 + u_1 \\
                & 1 & x_3 \\
                & & 1
            \end{psmallmatrix} 
            \begin{psmallmatrix}
                -1 & & \\
                & & -1 \\
                & -1 & n
            \end{psmallmatrix} l_U g\right) \overline{\e(m' x_1)} \, d^3x \, du_1 \\
        &= \sum_{n} \int E\left(\chi_\text{min},
            \begin{psmallmatrix}
                1 & x_1 & x_2 \\
                & 1 & x_3 \\
                & & 1
            \end{psmallmatrix} 
            \begin{psmallmatrix}
                -1 & & \\
                & & -1 \\
                & -1 & n
            \end{psmallmatrix} l_U g\right) \overline{\e(m' x_1 - m' n u_1)} \, d^3x \, du_1  
    \end{split}
\end{equation}
where we have made the substitutions $x_1 + n u_1 \to x_1$ and $x_2 + u_1 \to u_1$. Since the only $u_1$ dependence is in the character, the integral over $u_1$ picks up only the constant term $n = 0$. 

Thus, 
\begin{equation}
    \label{eq:SL3-FU-Whittaker}
    \begin{split}
        \MoveEqLeft
        F_U(\chi_\text{min}, m_1, m_2; g) = F_U(\chi_\text{min}, 0, m'; l_U g) \\
        &= \int E\left(\chi_\text{min},
            \begin{psmallmatrix}
                1 & x_1 & x_2 \\
                & 1 & x_3 \\
                & & 1 
            \end{psmallmatrix} 
            \begin{psmallmatrix}
                -1 & & \\
                & & -1 \\
                & -1 &
            \end{psmallmatrix} l_U g\right) \overline{\e(m' x_1)} \, d^3x \\
        &= W_{\psi_{m',0}}\Big(\chi_\text{min},
        \underbrace{
            \begin{psmallmatrix}
                -1 & & \\
                & & -1 \\
                & -1 &
            \end{psmallmatrix} l_U}_{l_U'} g\Big)
    \end{split}
\end{equation}

Hence, the extra integral over $u_1$ picks up a single Whittaker vector from the sum in $F_{\Oh_\text{min}}$ and for the case $m_1 = m_1' \neq 0$ and $m_2 = 0$ we have from \eqref{eq:SL-FU-l} that
\begin{equation}
    \label{eq:SL3-Fourier-Whitt-agree}
    F_U(\chi_\text{min}, m_1', 0; g) = W_{\psi_{m_1', 0}}(\chi_\text{min}, g) \, .
\end{equation}

To summarize, in the case where only $m_1$ is charged the Fourier coefficient completely agrees with the Whittaker vector charged under the same simple root, while the more general Fourier coefficient with $m_2 \neq 0$ is the same Whittaker vector but where the argument is translated with the Levi element $l_U'$.

Alternatively, one could obtain \eqref{eq:SL3-FU-Whittaker} by directly expanding $F_U(\chi, m_1, m_2; g)$ in the same way we expanded $F_{\Oh_\text{min}}$, but since we have expressed the full expansion \eqref{GeneralFE} of $E(\chi, g)$ in terms of orbit coefficients \eqref{eq:SL3-Whittaker-as-orbits} and since only the minimal orbit is non-vanishing for the minimal repesentation, it is more efficient, in the long run, to expand only $F_{\Oh_\text{min}}$ once.

We will also see how this workflow of connecting Fourier coefficients to orbit coefficients and then expanding the latter is easily used to obtain a generalization for the next-to-minimal representation in the next section where we will treat $SL(4)$.

\section{\texorpdfstring{$SL(4)$}{SL(4)} Eisenstein series}
\label{sec:SL4}
A general Eisenstein series $E(\chi, g)$ on $G(\ads) = SL(4, \ads)$ can be expanded as
\begin{equation}
    \label{eq:SL4-Whit-expansion}
    E(\chi, g) = \suml[\psi_N] W_N(\chi, \psi_N; g) + \suml[\psi_{N'} \neq 1] W_{N'}(\chi, \psi_{N'}; g)
\end{equation}
where $N' = [N, N]$. Note that we have included the constant term $E_0(\chi, g) = W_N(\chi, 1; g)$ in the first sum, corresponding to the term $\psi_N=1$. 

For convenience we will specify the Whittaker vectors on $N$ and $N'$ with a list of charges as follows
\begin{equation}
    \begin{split}
        W_N(\chi, m_1, m_4, m_6; g) &= \int E(\chi,
        \begin{psmallmatrix}
            1 & x_1 & x_2 & x_3 \\
            & 1 & x_4 & x_5 \\
            & & 1 & x_6 \\
            & & & 1
        \end{psmallmatrix}
        g ) \overline{\e(m_1 x_1 + m_4 x_4 + m_6 x_6)} \, d^6x \\
        W_{N'}(\chi, m_1, m_2, m_3; g) &= \int E(\chi,
        \begin{psmallmatrix}
            1 & & x_1 & x_2 \\
            & 1 & & x_3 \\
            & & 1 & \\
            & & & 1
        \end{psmallmatrix}
        g ) \overline{\e(m_1 x_1 + m_2 x_2 + m_3 x_3)} \, d^3x \, .
    \end{split}
\end{equation}

\subsection{Fourier coefficients attached to nilpotent orbits}

The nilpotent orbits of $SL(n)$ can be identified with partitions of $n$ written with multiplicities in superscript, that is, $(n_1^{k_1} \ldots n_m^{k_m})$ where $k_1 n_1 + \cdots + k_m n_m = n$ \cite{CollingwoodMcGovern}. The orbits for $SL(4)$ are shown in table \ref{tab:SL4-orbits}.

\newcommand{\SLFourDynkin}[3]
{
    \scalebox{0.75}{
    \begin{tikzpicture}[baseline={([yshift=-.8ex]current bounding box.center)}]
    \node(n1) at (0, 0) [dot] {};
    \node(n2) at (1, 0) [dot] {};
    \node(n3) at (2, 0) [dot] {};

    \node at ($ (n1) + (0, 0.5) $) {#1};
    \node at ($ (n2) + (0, 0.5) $) {#2};
    \node at ($ (n3) + (0, 0.5) $) {#3};

    \begin{scope}[on background layer]
        \draw [line] (n1) -- (n2);
        \draw [line] (n2) -- (n3);
    \end{scope}
    \end{tikzpicture}
    }
}

\begin{table}[ht]
    \tabulinesep=0.5em
    \centering
    \topcaption{\label{tab:SL4-orbits}Nilpotent $SL(4)$ orbits with weighted Dynkin diagrams, dimensions, stabilizer type $C$ and $V = U_{i \geq 2}$ \cite{CollingwoodMcGovern, Carter}. The labelling of the orbit is done by both the partition and the corresponding Bala--Carter label.} 
    \begin{tabu}{r@{\hspace{3pt}}lcccc}
        \toprule
        \multicolumn{2}{c}{Orbit $\mathcal{O}$} & Weighted Dynkin diagram & $\dim(\mathcal{O})$ & $C$ & $V$ \\
        \midrule
        $(4)$ &$=A_3$ & \SLFourDynkin{2}{2}{2} & $12$ & $\id$ &
        \def\arraystretch{1.0}
        $\begin{psmallmatrix}
            1 & * & * & * \\
            & 1 & * & * \\
            &   & 1 & * \\
            &   &   & 1 \\
        \end{psmallmatrix}$ \\
        $(31)$ &$=A_2$ & \SLFourDynkin{2}{0}{2} & $10$ & $T_1$ &
        \def\arraystretch{1.0}
        $\begin{psmallmatrix}
            1 & * & * & * \\
            & 1 &  & * \\
            &   & 1 & * \\
            &   &   & 1 \\
        \end{psmallmatrix}$ \\
        $(2^2)$ &$=2A_1$ & \SLFourDynkin{0}{2}{0} & $8$ & $A_1$ &
        \def\arraystretch{1.0}
        $\begin{psmallmatrix}
            1 &   & * & * \\
            & 1 & * & * \\
            &   & 1 &   \\
            &   &   & 1 \\
        \end{psmallmatrix}$ \\
        $(21^2)$ &$=A_1$ & \SLFourDynkin{1}{0}{1} & $6$ & $A_1 \times T_1$ &
        \def\arraystretch{1.0}
        $\begin{psmallmatrix}
            1 &   &   & * \\
            & 1 &   &   \\
            &   & 1 &   \\
            &   &   & 1 \\
        \end{psmallmatrix}$ \\ 
        $(1^4)$ &$=0$ & \SLFourDynkin{0}{0}{0} & $0$ & $A_3$ &
        \def\arraystretch{1.0}
        $\begin{psmallmatrix}
            1 &   &   &   \\
            & 1 &   &   \\
            &   & 1 &   \\
            &   &   & 1 \\
        \end{psmallmatrix}$ \\
        \bottomrule
    \end{tabu}
\end{table}
\def\arraystretch{1}

Recall that for a Fourier coefficient on $V = N$ to be attached to the regular orbit $(4)$, the character $\psi_N$ needs to be non-degenerate (generic) to be stabilized only by the trivial $C = \id$.

We will now introduce a convenient, concrete description of the characters based on the identification with $\lie u^*$ discussed at the end of section \ref{sec:fourier-coeffs}.

A character $\psi_U$ on a general unipotent group $U$ of $SL(n)$ can be described by a matrix $\M$ in $U(\rats)/[U,U](\rats)$ with elements $\M_{ij}$ by 
\begin{equation}
    \psi_U(u) = \e(\tr \M^T u) = \e\Big(\sum_{ij} \M_{ij} u_{ij}\Big)
\end{equation}
We will call such a matrix $\M$ a \emph{character matrix} and denote the corresponding Fourier coefficient on $U$ by $F_U(\M; g)$.

The $L(\rats)$-conjugation of a character in \eqref{eq:L-conjugation} can then expressed as
\begin{equation}
    \label{eq:charge-matrix-conjugation}    
    F_U(\M; g) = F_U(\M'; lg) \qquad \M' = (l^T)^{-1} \M l^T
\end{equation}
for $l \in L(\rats)$.

Since also $l^T \in L(\rats)$ for $SL(n)$ the character variety orbits discussed in section \ref{sec:fourier-coeffs} can be described as $L(\rats)$-orbits of the character matrices $\M$. We will also discuss $G(\rats)$-orbits of the character matrices in section \ref{sec:SL4-parabolic-coeffs}.

An orbit coefficient of section \ref{sec:orbit-rep} may be similarly described by a matrix $\M$ in $V(\rats)/[V,V](\rats)$ and we are now interested in which such matrices give non-zero orbit coefficients. According to section \ref{sec:orbit-rep} the stabiliser $C(\psi_V^\Oh) \subset L_\Oh$ of $\psi_V^\Oh$ must be of the same type as that of $\Oh$ as listed in table \ref{tab:SL4-orbits}.

If $\M \in V(\rats)/[V,V](\rats)$ is in the $G$-orbit $\Oh$ then $C_G(H) = L_\Oh$ and the type of $C_G(M) \cap L_\Oh$ is trivially the same as the stabilizer type associated to the orbit $\Oh$ as defined in section \ref{nilpotentorbits} and the Fourier coefficient with character $\psi_V^\Oh$ described by $\M$ is attached to $\Oh$.

The condition for $\M$ to be in the orbit $\Oh$ can be found by requiring that the corresponding Lie algebra element $X$, where $\M = \exp(X)$ can form a Jacobson-Morozov triple with $H$ given by the weighted Dynkin diagram.

\pagebreak

\begin{example}[Conditions for $F_{(2^2)}$]~\\
    From the conjugacy class of Jacobson-Morozov triples associated to our nilpotent orbit it is possible to choose a representative of the neutral element $H$ such that it belongs to our Cartan subalgebra $\lie h$ \cite{CollingwoodMcGovern}.

    As seen from the weighted Dynkin diagram in table \ref{tab:SL4-orbits} the Cartan element associated to the nilpotent orbit $(2^2)$ satisfies $[H, E_{\alpha_1}] = 0$, $[H, E_{\alpha_2}] = 2 E_{\alpha_2}$ and $[H, E_{\alpha_3}] = 0$, which, together with tracelessness gives $H = \diag(1, 1, -1, -1)$.

    The conditions $[H, X] = 2 X$ and $[H, Y] = -2 Y$ requires $X$ and $Y$ to be of the forms
    \begin{equation}
        X = 
        \begin{psmallmatrix}
            0 & 0 & x_1 & x_2 \\
            0 & 0 & x_3 & x_4 \\
            0 & 0 & 0 & 0 \\
            0 & 0 & 0 & 0
        \end{psmallmatrix} \qquad
        Y =
        \begin{psmallmatrix}
            0 & 0 & 0 & 0 \\
            0 & 0 & 0 & 0 \\
            y_1 & y_3 & 0 & 0 \\
            y_2 & y_4 & 0 & 0
        \end{psmallmatrix} \qquad
    \end{equation}

    Then, one can show that the condition that there should exist a $Y$ on this form such that $[X, Y] = H$, is equivalent to that the elements of $X$ satisfy $x_1 x_4 - x_2 x_3 \neq 0$. We can now relate this condition to the elements of $\M$ by
    \begin{equation}
        \M = \exp(X) = 
        \begin{psmallmatrix}
            1 & 0 & x_1 & x_2 \\
            0 & 1 & x_3 & x_4 \\
            0 & 0 & 1 & 0 \\
            0 & 0 & 0 & 1
        \end{psmallmatrix} \in V(\rats)/[V,V](\rats)
    \end{equation}
\end{example}

The remaining orbits can be treated similarly. The Fourier coefficients attached to the different $SL(4)$ orbits together with the conditions on the charges are shown in table \ref{tab:SL4-orbit-coefficients}. 

\begin{table}[!htbp]
    \centering
    \topcaption{\label{tab:SL4-orbit-coefficients}$SL(4)$ orbit coefficients with character matrix and conditions on the charges.}
    \vspace{-1.75em}
    $$\begin{array}{lc}
        \toprule
        \multicolumn{1}{c}{\text{Fourier coefficient}} & \text{Character matrix} \\ 
        \midrule
        \begin{aligned}
            \MoveEqLeft
            F_{(4)}(\chi, m_1, m_4, m_6; g) \\
            &= \int E(\chi,
                \begin{psmallmatrix}
                    1 & x_1 & x_2 & x_3 \\
                    & 1 & x_4 & x_5 \\
                    & & 1 & x_6 \\
                    & & & 1
                \end{psmallmatrix}
                g) \overline{\e(m_1 x_1 + m_4 x_4 + m_6 x_6)} \, d^6x
        \end{aligned}
        &
        \begin{gathered}
            \begin{psmallmatrix}
            1 & m_1 & & \\
            & 1 & m_4 & \\
            & & 1 & m_6 \\
            & & & 1
            \end{psmallmatrix}
            \\
            m_1 m_4 m_6 \neq 0
        \end{gathered} \\[5em]
        \begin{aligned}
            \MoveEqLeft
            F_{(31)}(\chi, m_1, m_2, m_5, m_6; g) \\
            & = \int E(\chi,
                \begin{psmallmatrix}
                    1 & x_1 & x_2 & x_3 \\
                    & 1 & & x_5 \\
                    & & 1 & x_6 \\
                    & & & 1
                \end{psmallmatrix}
                g) \overline{\e(m_1 x_1 + m_2 x_2 + m_5 x_5 + m_6 x_6)} \, d^5x \quad
        \end{aligned}
        & 
        \begin{gathered}
            \begin{psmallmatrix}
                1 & m_1 & m_2 & \\
                & 1 & & m_5 \\
                & & 1 & m_6 \\
                & & & 1
            \end{psmallmatrix}
            \\ 
            m_1 m_5 + m_2 m_6 \neq 0 
        \end{gathered} \\[5em]
        \begin{aligned}
            \MoveEqLeft
            F_{(2^2)}(\chi, m_1, m_2, m_3, m_4; g) \\
            & = \int E(\chi,
                \begin{psmallmatrix}
                    1 & & x_1 & x_2 \\
                    & 1 & x_3 & x_4 \\
                    & & 1 & \\
                    & & & 1
                \end{psmallmatrix}
                g) \overline{\e(m_1 x_1 + m_2 x_2 + m_3 x_3 + m_4 x_4)} \, d^4x
        \end{aligned} 
        & 
        \begin{gathered}
            \begin{psmallmatrix}
                1 & & m_1 & m_2 \\
                & 1 & m_3 & m_4 \\
                & & 1 & \\
                & & & 1
            \end{psmallmatrix}
            \\
            m_1 m_4 - m_2 m_3 \neq 0
        \end{gathered}
        \\[5em]
        \begin{aligned}
            \MoveEqLeft
            F_{(21^2)}(\chi, m_1; g) \\
            & = \int E(\chi,
                \begin{psmallmatrix}
                    1 & & & x_1 \\
                    & 1 & & \\
                    & & 1 & \\
                    & & & 1
                \end{psmallmatrix}
                g) \overline{\e(m_1 x_1)} \, dx_1
        \end{aligned}
        &
        \begin{gathered}
            \begin{psmallmatrix}
                1 & & & m_1 \\
                & 1 & & \\
                & & 1 & \\
                & & & 1   
            \end{psmallmatrix}
            \\
            m_1 \neq 0 
        \end{gathered}    
        \\[5em]
        \begin{aligned}
            \MoveEqLeft
            F_{(1^4)}(\chi; g) \\
            & = \int E(\chi,
                \begin{psmallmatrix}
                    1 & x_1 & x_2 & x_3 \\
                    & 1 & x_4 & x_5 \\
                    & & 1 & x_6 \\
                    & & & 1
                \end{psmallmatrix}
                g) \, d^6x = E_0(\chi, g)
        \end{aligned}
        & \\
        \bottomrule
    \end{array}$$
\end{table}

\clearpage

\subsection{Expansion in nilpotent orbits and Theorem \ref{thm:orbit-exp} for \texorpdfstring{$SL(4)$}{SL(4)}}
Similar to how the Whittaker vectors were related to orbit coefficients in section \ref{sec:SL3-orbit-expansion} for $SL(3)$, the Whittaker vectors for $SL(4)$ can be expressed in terms of orbit coefficients as shown in table \ref{tab:SL4-Whittaker-as-orbits}. See appendix \ref{app:SL4-Whittaker-as-orbits} for further details.

\begin{table}[!htbp]
\centering
    \topcaption{\label{tab:SL4-Whittaker-as-orbits}$SL(4)$ Whittaker vectors in terms of orbit coefficients. Primed charges are non-zero and the undecorated charges are arbitrary.}
    \vspace{-1.75em}
    \begin{align*}
        \toprule
        \multicolumn{1}{c}{\text{Whittaker vector}} & \multicolumn{1}{c}{\text{Orbit coefficient}} \\
        \midrule
        W_N(\chi, 0, 0, 0; g) &= F_{(1^4)}(\chi, g) \displaybreak[0]\\[1em]
        W_N(\chi, m_1', 0, 0; g) &=
        \int F_{(21^2)}(\chi, m_1'; 
        \begin{psmallmatrix}
            1 & & & \\
            & & & -1 \\
            & & 1 & \\
            & 1 & &
        \end{psmallmatrix}
        \begin{psmallmatrix}
            1 & & u_2 & u_3 \\
            & 1 & u_4 & u_5 \\
            & & 1 & u_6 \\
            & & & 1
        \end{psmallmatrix}
        g) \, d^5u \\
        W_N(\chi, 0, m_4', 0; g) &=
        \int F_{(21^2)}(\chi, m_4'; 
        \begin{psmallmatrix}
            & 1 & & \\
            1 & & & \\
            & & & 1 \\
            & & 1 &
        \end{psmallmatrix}
        \begin{psmallmatrix}
            1 & u_1 & u_2 & u_3 \\
            & 1 & & u_5 \\
            & & 1 & u_6 \\
            & & & 1
        \end{psmallmatrix}
        g) \, d^5u \\
        W_N(\chi, 0, 0, m_6'; g) &=
        \int F_{(21^2)}(\chi, m_6'; 
        \begin{psmallmatrix}
            & & 1 & \\
            & 1 & & \\
            -1 & & & \\
            & & & 1
        \end{psmallmatrix}
        \begin{psmallmatrix}
            1 & u_1 & u_2 & u_3 \\
            & 1 & u_4 & u_5 \\
            & & 1 & \\
            & & & 1
        \end{psmallmatrix}
        g) \, d^5u
        \displaybreak[0]\\[1em]
        W_N(\chi, m_1', m_4', 0; g) &=
        \sum_m \int F_{(31)}(\chi, m_1', 0, m_4', m;
        \begin{psmallmatrix}
            1 & & & \\
            & 1 & & \\
            & & & -1 \\
            & & 1 &
        \end{psmallmatrix}
        \begin{psmallmatrix}
            1 & & & \\
            & 1 & & u_5 \\
            & & 1 & u_6 \\
            & & & 1
        \end{psmallmatrix}
        g) \, d^2 u \\
        W_N(\chi, 0, m_4', m_6'; g) &=
        \sum_m \int F_{(31)}(\chi, m, m_4', 0, m_6';
        \begin{psmallmatrix}
            & 1 & & \\
            -1 & & & \\
            & & 1 & \\
            & & & 1
        \end{psmallmatrix}
        \begin{psmallmatrix}
            1 & u_1 & u_2 & \\
            & 1 & & \\
            & & 1 & \\
            & & & 1
        \end{psmallmatrix}
        g) \, d^2 u \\
        W_N(\chi, m_1', 0, m_6'; g) &=
        \sum_m \int F_{(2^2)}(\chi, -m_1', 0, m, m_6'; 
        \begin{psmallmatrix}
            -1 & & & \\
            & & 1 & \\
            & 1 & & \\
            & & & 1
        \end{psmallmatrix}
        \begin{psmallmatrix}
            1 & & u_2 & \\
            & 1 & u_4 & u_5 \\
            & & 1 & \\
            & & & 1
        \end{psmallmatrix}
        g) \, d^3u
        \displaybreak[0]\\[1em]
        W_N(\chi, m_1', m_4', m_6'; g) &= F_{(4)}(\chi, m_1', m_4', m_6'; g) \\[1em]
        W_{N'}(\chi, m_1', 0, 0; g) &= \int F_{(21^2)}(\chi, m_1'; 
        \begin{psmallmatrix}
            1 & & & \\
            & 1 & & \\
            & & & -1 \\
            & & 1 &
        \end{psmallmatrix}
        \begin{psmallmatrix}
            1 & & & u_2 \\
            & 1 & & u_3 \\
            & & 1 & \\
            & & & 1
        \end{psmallmatrix}
        g) \, d^2u \\
        W_{N'}(\chi, 0, 0, m_3'; g) &= \int F_{(21^2)}(\chi, m_3'; 
        \begin{psmallmatrix}
            & 1 & & \\
            -1 & & & \\
            & & 1 & \\
            & & & 1
        \end{psmallmatrix}
        \begin{psmallmatrix}
            1 & & u_1 & u_2 \\
            & 1 & & \\
            & & 1 & \\
            & & & 1
        \end{psmallmatrix}
        g) \, d^2u \\
        W_{N'}(\chi, m_1', 0, m_3'; g) &= \sum_m F_{(2^2)}(\chi, m_1', 0, m, m_3'; g) \\
        W_{N'}(\chi, m_1, m_2', m_3; g) &= \int F_{(21^2)}(\chi, m_2'; 
        \begin{psmallmatrix}
            1 & & u_1 & \\
            & 1 & & u_3 \\
            & & 1 & \\
            & & & 1
        \end{psmallmatrix}
        \begin{psmallmatrix}
            1 & m_3/m_2' & & \\
            & 1 & & \\
            & & 1 & -m_1/m_2' \\
            & & & 1
        \end{psmallmatrix}
        g) \, d^2u \\
        \bottomrule
    \end{align*}
\end{table}

\clearpage

\begin{proof}[\textbf{Proof of Theorem \ref{thm:orbit-exp} for $SL(4)$}]~\par
    Table \ref{tab:SL4-Whittaker-as-orbits} together with the expansion in \eqref{eq:SL4-Whit-expansion} proves that we can expand $E(\chi, g)$ in terms of orbit coefficients.

    That is, regrouping \eqref{eq:SL4-Whit-expansion}, we have that
    \begin{equation}
        E(\chi, g) = \mathcal F_{(1^4)}(\chi, g) + \mathcal F_{(21^2)}(\chi, g) + \mathcal F_{(2^2)}(\chi, g) + \mathcal F_{(31)}(\chi, g) + \mathcal F_{(4)}(\chi, g)
    \end{equation}
    with
    \begin{equation}
        \begin{split}
            \mathcal F_{(1^4)}(\chi, g) &= E_0(\chi, g) \\
            \mathcal F_{(21^2)}(\chi, g) &= \sum_{m' \neq 0} \Big( W_N(\chi, m', 0, 0, 0; g) + W_N(\chi, 0, m', 0; g) + W_N(\chi, 0, 0, m'; g) +{} \\
           & \qquad \qquad + W_{N'}(\chi, m', 0, 0; g) + W_{N'}(\chi, 0, 0, m'; g) + \sum_{n,k} W_{N'}(\chi, n, m', k; g) \Big) \\
           \mathcal F_{(2^2)}(\chi, g) &= \sum_{m', n' \neq 0} \Big( W_N(\chi, m', 0, n'; g) + W_{N'}(\chi, m', 0, n'; g) \Big) \\
           \mathcal F_{(31)}(\chi, g) &= \sum_{m',n' \neq 0} \Big( W_N(\chi, m', n', 0; g) + W_N(\chi, 0, m', n'; g) \Big) \\
           \mathcal F_{(4)}(\chi, g) &= \sum_{m', n', k' \neq 0} W_N(\chi, m', n', k'; g)
        \end{split}
    \end{equation}
    where each $\mathcal F_\Oh$ is linearly determined by $F_\Oh$ coefficients as seen in table \ref{tab:SL4-Whittaker-as-orbits}.
\end{proof}

\subsection{The minimal representation and Theorem \ref{thm:min-rep} for \texorpdfstring{$SL(4)$}{SL(4)}}
We are now ready to prove Theorem \ref{thm:min-rep} for $SL(4)$, but first let us rephrase it using the language introduced in this section.

\begin{theorem}[Reformulated Theorem \ref{thm:min-rep} for $SL(4)$]
    \label{thm:SL4-min-rep}
  Assume $E(\chi, g)$ is in the minimal representation. Then the Fourier coefficients attached to the orbits $\Oh \notin \overline{\Oh_\text{min}}$, that is, the orbits $(4)$, $(31)$ and $(2^2)$, vanish and the Fourier coefficients attached to $\Oh_\text{min} = (21^2)$ can then be expressed as a sum of maximally degenerate Whittaker vectors on $N$.
\end{theorem}

For the proof we will need the following lemma. This step can be done in several ways resulting in different sums of maximally degenerate Whittaker vectors on $N$. Below we show a version resulting in Whittaker vectors supported only on $\alpha_2$. In appendix \ref{app:SL4-F211-alternative-expansion} versions for $\alpha_1$ and $\alpha_3$ are shown as well.

\begin{lemma}
    \label{lem:SL4-211-expanded}
    For any representation, the $F_{(21^2)}$ coefficient can be expressed as
    \begin{equation}
        \label{eq:SL4-211-expanded}
        \begin{split}
            \MoveEqLeft
            F_{(21^2)}(\chi, m_4'; g) \\
                &= \sum_{l \in L_{(22)}^*} \sum_{m_3' \neq 0} F_{(2^2)}(\chi, 0, m_3', m_4', 0; lg) + \sum_{l \in L_{(22)}^*} \sum_{m_1, m_6} W_N(\chi, m_1, m_4', m_6; l g) 
        \end{split}     
    \end{equation}
    where $L_{(2^2)}^*$ is a certain subset of $L_{(2^2)}(\rats)$ defined below.
\end{lemma}

\begin{proof}
    The integrand of the Fourier coefficient $F_{(21^2)}$ can be expanded further as
    \begin{equation}
        \begin{split}
            \MoveEqLeft
            F_{(21^2)}(\chi, m_4'; g) \\
            &= \int E(\chi,
            \begin{psmallmatrix}
                1 & & & x_4 \\
                & 1 & & \\
                & & 1 & \\
                & & & 1
            \end{psmallmatrix}
            g) \overline{\e(m_4' x_4)} \, dx_4 \\
            &= \sum_{m_2, m_3, m_5}\int E(\chi,
            \begin{psmallmatrix}
                1 & & x_5 & x_4 \\
                & 1 & x_3 & x_2 \\
                & & 1 & \\
                & & & 1
            \end{psmallmatrix}
            g) \overline{\e(m_2 x_2 + m_3 x_3 + m_4' x_4 + m_5 x_5)} \, d^4x 
        \end{split}
    \end{equation}
    
    Conjugating and redefining the parameters in the prefactor of $g$ (see \eqref{eq:charge-matrix-conjugation}) we obtain
    \begin{equation}
        \begin{split}
            \sum_{m_2, m_3, m_5} \int E(\chi,
            \begin{psmallmatrix}
                1 & & x_2 & x_3 \\
                & 1 & x_4 & x_5 \\
                & & 1 & \\
                & & & 1
            \end{psmallmatrix}
            \begin{psmallmatrix}
                & 1 & & \\
                1 & m_2/m_4' & & \\
                & & & 1 \\
                & & 1 & -m_5/m_4'
            \end{psmallmatrix}
            g) \overline{\e(m_4' x_4 + (m_3 - \tfrac{m_2 m_5}{m_4'}) x_3)} \, d^4x
        \end{split}
    \end{equation}
    When $m_3' = m_3 - \tfrac{m_2 m_5}{m_4'} \neq 0$ the coefficient is attached to the orbit $(2^2)$ as seen in table \ref{tab:SL4-orbit-coefficients}. 
    
    Let us rewrite the sum over $m_2$ and $m_5$ as a sum over elements in $L_{(2^2)}$ by defining the subset
    \begin{equation}
        L_{(2^2)}^* = \left\{
        \begin{psmallmatrix}
            & 1 & & \\
            1 & a & & \\
            & & & 1 \\
            & & 1 & b
        \end{psmallmatrix}
        : a, b \in \rats
        \right\}
    \end{equation} 
    where we have noted that we can scale away the $m_4'$ dependence.

    We thus have
    \begin{equation}
        \begin{split}
            \MoveEqLeft
            F_{(21^2)}(\chi, m_4'; g) \\
                &= \sum_{l \in L_{(22)}^*} \sum_{m_3' \neq 0} F_{(2^2)}(\chi, 0, m_3', m_4', 0; lg) + \sum_{l \in L_{(22)}^*} \int E(\chi,
                \begin{psmallmatrix}
                    1 & & x_2 & x_3 \\
                    & 1 & x_4 & x_5 \\
                    & & 1 & \\
                    & & & 1
                \end{psmallmatrix}
                lg) \overline{\e(m_4' x_4)} \, d^4x
        \end{split}
    \end{equation}
    The last term we expand further as
    \begin{equation}
        \begin{split}
            \MoveEqLeft
            \sum_{l \in L_{(22)}^*} \sum_{m_1, m_6} \int E(\chi,
            \begin{psmallmatrix}
                1 & x_1 & x_2 & x_3 \\
                & 1 & x_4 & x_5 \\
                & & 1 & x_6 \\
                & & & 1
            \end{psmallmatrix}
            lg) \overline{\e(m_1 x_1 + m_4' x_4 + m_6 x_6)} \, d^4x \\
            &= \sum_{l \in L_{(22)}^*} \sum_{m_1, m_6} W_N(\chi, m_1, m_4', m_6; l g)
        \end{split}
    \end{equation}
    where we have made the substitutions $x_2 + x_1 x_4 \to x_2$ and $x_3 + x_1 x_5 \to x_3$.
\end{proof}

We note that, the last term can be connected to different orbit coefficients using table \ref{tab:SL4-Whittaker-as-orbits}. That is, $m_1, m_6 \neq 0$ is connected to the regular orbit $(4)$, $m_1 \neq 0, m_6 = 0$ and $m_1 = 0, m_6 \neq 0$ to the orbit $(31)$, and $m_1 = m_6 = 0$ to the minimal orbit $(21^2)$. In this way we see how $F_{(21^2)}$ changes for different representations. 

\begin{proof}[\textbf{Proof of Theorem \ref{thm:SL4-min-rep}}]~\par
    Using the arguments of our section \ref{sec:orbit-rep} in the spirit of Matumoto and Mœglin-Waldspurger we know that in the minimal representation the only non-vanishing orbit coefficients are those attached to $(21^2)$ and $(1^4)$.
    
    The $(21^2)$ coefficient was further expanded into Fourier coefficients attached to orbits outside the closure of $(21^2)$ together with Whittaker vectors on $N$ in \eqref{eq:SL4-211-expanded} from lemma~\ref{lem:SL4-211-expanded}.
    
    From table \ref{tab:SL4-Whittaker-as-orbits} we know which orbit coefficients the different Whittaker vectors are connected with and using the arguments of section \ref{sec:orbit-rep} we know that the only surviving part of $F_{(21^2)}$ in the minimal representation is
    \begin{equation}
        F_{(21^2)}(\chi_\text{min}, m'; g) = \sum_{l \in L_{(22)}^*} W_N(\chi_\text{min}, 0, m', 0; l g)
    \end{equation}
\end{proof}

Thus, we have now proved that, in the minimal representation, $E(\chi, g)$ can be expressed as a sum of maximally degenerate Whittaker vectors on $N$ generalizing the results of Miller-Sahi \cite{MillerSahi} for $SL(4)$ ($SL(3)$ was shown above).

We note that, in the minimal representation, we may express $E(\chi, g)$ in terms of $F_{(21^2)}$ coefficients or maximally degenerate Whittaker vectors; the difference between them, as seen in lemma \ref{lem:SL4-211-expanded}, is higher orbit coefficients. We choose the latter since maximally degenerate Whittaker vectors are easier to compute.

This is also the source of the ambiguity of the $\mathcal F_\Oh$ terms discussed in the introduction; we may add higher orbit terms to $\mathcal F_\Oh$ which do not contribute when in the representation associated to $\Oh$.

\subsection{The next-to-minimal representation and Theorem \ref{thm:ntm-rep} for \texorpdfstring{$SL(4)$}{SL(4)}}
Before discussing the next-to-minimal representation, we note that it is only the sum $\sum_m F_{(2^2)}(m_1', 0, m, m_4'; g)$ that occurs in a full expansion according to table \ref{tab:SL4-Whittaker-as-orbits} and never $F_{(2^2)}$ alone. It turns out that we need to consider this sum instead of separate coefficients to be able to expand it in a similar way as for the $(21^2)$ coefficient. Therefore we will make the following definition 
\begin{equation}
    \label{eq:SL4-22-partially-summed-form}
    \begin{split}
        \MoveEqLeft
        F_{[2^2]}(\chi, m_1', m_4'; g) \\
        &:= \sum_m F_{(2^2)}(m_1', 0, m, m_4'; g) = \int E(\chi, 
        \begin{psmallmatrix}
            1 & & x_1 & x_2 \\
            & 1 & & x_4 \\
            & & 1 & \\
            & & & 1
        \end{psmallmatrix}
        g) \overline{\e(m_1' x_1 + m_4' x_4)} \, d^3 x
    \end{split}
\end{equation}
which we will call a \emph{partially summed} $(2^2)$ coefficient or a Fourier coefficient attached to $(2^2)$ in a \emph{partially summed form}. We will now show that an ordinary $(2^2)$ coefficient can be expressed in terms of a partially summed coefficient.

\begin{lemma}
    \label{lem:SL4-22-partially-summed-form}
    All coefficients attached to the orbit $(2^2)$ can be expressed on the form
    \begin{equation}
        \sum_{m} \int F_{(2^2)}(m_1', 0, m, m_4';
        \begin{psmallmatrix}
            1 & & & \\
            & 1 & u & \\
            & & 1 & \\
            & & & 1 
        \end{psmallmatrix}
        lg) \, du 
    \end{equation}
    for some $l \in L_{(2^2)}(\rats)$.
\end{lemma}
\begin{proof}
    For $F_{(2^2)}(m_1, m_2, m_3, m_4; g)$ it is required that $m_1 m_4 - m_2 m_3 \neq 0$. This implies that at least one of $m_1$ and $m_3$ is non-zero. Since $m_1$ and $m_3$ can be interchanged with a Weyl reflection which is in $L_{(2^2)}(\rats)$ we can assume, without loss of generality that $m_1 = m_1' \neq 0$.

    Using conjugation we then have that
    \begin{gather}
        F_{(2^2)}(m_1', m_2, m_3, m_4; g) = F_{(2^2)}(m_1', 0, 0, m_4'; l g) \\
        m_4' = m_4 - \frac{m_2 m_3}{m_1'} \neq 0 \qquad l = 
        \begin{psmallmatrix}
            1 & m_3/m_1' & & \\
            & 1 & & \\
            & & 1 & \\
            & & -m_2/m_1' & 1
        \end{psmallmatrix} \in L_{(2^2)}(\rats)
    \end{gather}

    This coefficient can be expressed as
    \begin{equation}
        \begin{split}
            \MoveEqLeft
            F_{(2^2)}(m_1', 0, 0, m_4'; l g) \\
            &= \sum_{m_3} \int F_{(2^2)}(m_1', 0, m_3, m_4';
            \begin{psmallmatrix}
                1 & & & \\
                & 1 & u_3 & \\
                & & 1 & \\
                & & & 1 
            \end{psmallmatrix}
            lg) \, du_3 
        \end{split}
    \end{equation}
    since the integral over $u_3$ picks up only the $m_3 = 0$ term.
\end{proof}

\begin{remark}
    The sum of $F_{(2^2)}$ over $l$ in \eqref{eq:SL4-211-expanded} can be put in a partially summed form by
    \begin{equation}
        \begin{split}
            \MoveEqLeft
            \sum_{m, n} F_{(2^2)}(\chi, 0, m_3' , m_4', 0;
            \begin{psmallmatrix}
                & 1 & & \\
                1 & m/m_4' & & \\
                & & & 1 \\
                & & 1 & n/{m_4'}
            \end{psmallmatrix}
            g) \\
            &= \sum_{m, n} F_{(2^2)}(\chi, 0, m_3' , m_4', 0;
            \begin{psmallmatrix}
                & 1 & & \\
                -1 & m/m_4' & & \\
                & & 1 & \\
                & & & 1 
            \end{psmallmatrix}
            \begin{psmallmatrix}
                -1 & & & \\
                & 1 & & \\
                & & & 1 \\
                & & 1 & n/{m_4'}
            \end{psmallmatrix} g) \\
            &= \sum_{m, n} F_{(2^2)}(\chi, -m_4', 0, m, m_3';
            \begin{psmallmatrix}
                -1 & & & \\
                & 1 & & \\
                & & & 1 \\
                & & 1 & n/{m_4'}
            \end{psmallmatrix} g) \\
            &= \sum_n F_{[2^2]}(\chi, -m_4', m_3';
            \begin{psmallmatrix}
                -1 & & & \\
                & 1 & & \\
                & & & 1 \\
                & & 1 & n/{m_4'}
            \end{psmallmatrix} g)
        \end{split}
    \end{equation}
    which will be useful later.
\end{remark}{}

We are now ready to prove Theorem \ref{thm:ntm-rep} which we will now rephrase using the language introduced in this section.

\begin{theorem}[Reformulated Theorem \ref{thm:ntm-rep}]
    \label{thm:SL4-ntm-rep} 
    Assume $E(\chi, g)$ is in the next-to-minimal representation. Then the Fourier coefficients attached to the orbits $\Oh \notin \overline{\Oh_\text{ntm}}$, that is, $(4)$ and $(31)$ vanish. The Fourier coefficients attached to $\Oh_\text{ntm} = (2^2)$ can then be expressed in terms of Whittaker vectors $W_N(\chi, m, 0, n; g)$ with $m, n \neq 0$; more specifically, the partially summed $(2^2)$ coefficients defined in \eqref{eq:SL4-22-partially-summed-form} can be expressed as a sum of such Whittaker vectors at different translates of $g$.
   
    The coefficients attached to $(21^2)$ can, in the next-to-minimal representation, be expressed as a sum of maximally degenerate Whittaker vectors together with a sum of degenerate Whittaker vectors on the same form as for the partially summed $(2^2)$ coefficients.
\end{theorem}

For the proof we will need the following lemma.

\begin{lemma}
    \label{lem:SL4-22-expanded}
    For any representation, the partially summed $(2^2)$ orbit coefficient can be expressed as
    \begin{equation}
        \label{eq:SL4-22-expanded}
        \begin{split}
            \MoveEqLeft
            F_{[2^2]}(\chi, m_1', m_6'; g) \\
            &= \sum_{l \in L_{(31)}^*} \sum_{m_5' \neq 0} F_{(31)}(-m_1', 0, m_5', m_6'; lg) + \sum_{l \in L_{(31)}^*} \sum_{m_4} W_N(\chi, -m_1', m_4, m_6'; lg)
        \end{split}
    \end{equation}
    where $L_{(31)}^*$ is a certain subset of $L_{(31)}(\rats)$ defined below.
\end{lemma}

\begin{proof}
    The partially summed $(2^2)$ coefficient can be further expanded as
    \begin{equation}
        \begin{split}
            \MoveEqLeft
            F_{[2^2]}(\chi, m_1', m_6'; g) \\
            &= \int E(\chi, 
            \begin{psmallmatrix}
                1 & & x_1 & x_3 \\
                & 1 & & x_6 \\
                & & 1 & \\
                & & & 1
            \end{psmallmatrix}
            g) \overline{\e(m_1' x_1 + m_6' x_6)} \, d^3x \\
            &= \sum_{m_2, m_5} \int E(\chi,
            \begin{psmallmatrix}
                1 & x_2 & x_1 & x_3 \\
                & 1 & & x_6 \\
                & & 1 & x_5 \\
                & & & 1
            \end{psmallmatrix}
            g) \overline{\e(m_1' x_1 + m_2 x_2 + m_5 x_5 + m_6' x_6)} \, d^5x \\
            &= \sum_{m_2, m_5} \int E(\chi, 
            \begin{psmallmatrix}
                1 & x_1 & x_2 & x_3 \\
                & 1 & & x_5 \\
                & & 1 & x_6 \\
                & & & 1
            \end{psmallmatrix}
            \begin{psmallmatrix}
                -1 & & & \\
                & & 1 & \\
                & 1 & -\frac{m_2}{m_1'} & \\
                & & & 1
            \end{psmallmatrix}
            g) \overline{\e(-m_1' x_1 + (m_5 + \tfrac{m_2 m_6'}{m_1'}) x_5 + m_6' x_6)} \, d^5x
        \end{split}
    \end{equation}

    When $m_5' = m_5 + \tfrac{m_2 m_6'}{m_1'} \neq 0$ the coefficient is attached to the orbit $(31)$. 
    
    We rewrite the sum over $m_2$ as a sum over elements in $L_{(31)}$ by defining the subset
    \begin{equation}
        L_{(31)}^* = \left\{
            \begin{psmallmatrix}
                -1 & & & \\
                & & 1 &  \\
                & 1 & a & \\
                & & & 1
            \end{psmallmatrix} : a \in \rats
        \right\}   
    \end{equation}
    Expand the remaining term to obtain
    \begin{equation}
        \begin{split}
            \MoveEqLeft
            F_{[2^2]}(\chi, m_1', m_6'; g) \\
            &= 
            \sum_{l \in L_{(31)}^*} \sum_{m_5' \neq 0} F_{(31)}(-m_1', 0, m_5', m_6'; lg) +{} \\
            & \quad + \sum_{l \in L_{(31)}^*} \int E(\chi, 
            \begin{psmallmatrix}
                1 & x_1 & x_2 & x_3 \\
                & 1 & & x_5 \\
                & & 1 & x_6 \\
                & & & 1
            \end{psmallmatrix}
            lg) \overline{\e(-m_1' x_1 + m_6' x_6)} \, d^5x \\
            &= \sum_{l \in L_{(31)}^*} \sum_{m_5' \neq 0} F_{(31)}(-m_1', 0, m_5', m_6'; lg) + \sum_{l \in L_{(31)}^*} \sum_{m_4} W_N(\chi, -m_1', m_4, m_6'; lg)
        \end{split}
    \end{equation}
    
\end{proof}

Again, using table \ref{tab:SL4-Whittaker-as-orbits} we see that the different Whittaker vectors in the last sum above can be connected to different orbits. The terms with $m_4 \neq 0$ are connected to the regular orbit $(4)$, while $m_4 = 0$ is connected to the next-to-minimal orbit $(2^2)$.

\begin{proof}[\textbf{Proof of Theorem \ref{thm:SL4-ntm-rep}}]~\par
    From the arguments of section \ref{sec:orbit-rep} we know that, in the next-to-minimal representation, the only non-vanishing orbit coefficients are those attached to $(2^2)$, $(21^2)$ and $(1^4)$.

    Lemma \ref{lem:SL4-22-partially-summed-form} relates a general $(2^2)$ coefficients to its partially summed form which, in turn, can be expanded as in \eqref{eq:SL4-22-expanded} using lemma \ref{lem:SL4-22-expanded}.

    We know that if $E(\chi, g)$ is in the next-to-minimal representation the $F_{(31)}$ coefficient vanishes as well as the generic Whittaker vectors with $m_4 \neq 0$ according to table \ref{tab:SL4-Whittaker-as-orbits}.

    The remaining part of $F_{[2^2]}$ is then
    \begin{equation}
        F_{[2^2]}(\chi_\text{ntm}, m_1', m_6'; g) = \sum_{l \in L_{(31)}^*} W_N(\chi_\text{ntm}, -m_1', 0, m_6'; lg) 
    \end{equation}

    Similarly from lemma \ref{lem:SL4-211-expanded} and the remark after lemma \ref{lem:SL4-22-partially-summed-form} we have that
    \begin{equation}
        \begin{split}
            \MoveEqLeft
            F_{(21^2)}(\chi_\text{ntm}, m_4'; g) \\
            &= \sum_{m_3' \neq 0} \sum_b F_{[2^2]}(\chi_\text{ntm}, -m_4', m_3';
            \begin{psmallmatrix}
                -1 & & & \\
                & 1 & & \\
                & & & 1 \\
                & & 1 & b
            \end{psmallmatrix}
            g) + \sum_{a, b} W_N(\chi_\text{ntm}, 0, m_4', 0;
            \begin{psmallmatrix}
                & 1 & & \\
                1 & a & & \\
                & & & 1 \\
                & & 1 & b
            \end{psmallmatrix} g) \\
            &= \sum_{m_3' \neq 0} \sum_{a,b} W_N(\chi_\text{ntm}, m_4', 0, m_3';
            \begin{psmallmatrix}
                1 & & & \\
                & & & 1 \\
                & 1 & & a \\
                & & 1 & b
            \end{psmallmatrix} g)
            + \sum_{a, b} W_N(\chi_\text{ntm}, 0, m_4', 0;
            \begin{psmallmatrix}
                & 1 & & \\
                1 & a & & \\
                & & & 1 \\
                & & 1 & b
            \end{psmallmatrix} g) 
        \end{split}
    \end{equation}
\end{proof}

\subsection{Fourier coefficients on maximal parabolic subgroups}
\label{sec:SL4-parabolic-coeffs}
Now that we know how the orbit coefficients can be determined for the minimal and next-to-minimal representations, let us see how we can use this for other Fourier coefficients as well.

\begin{theorem}
    \label{thm:SL4-maximal-parabolic}
    Let $P_\alpha(\ads) = L_\alpha(\ads) U_\alpha(\ads)$ be the maximal parabolic subgroup of $SL(4, \ads)$ specified by the single simple root $\alpha$ in the Lie algebra of $U_\alpha$.
    Let $F_{\psi_{U_\alpha}}( g)$ be a Fourier coefficient on $U_\alpha$ of some automorphic form $\varphi$ on $SL(4,\ads)$ with a non-trivial character $\psi_U$ described by a character matrix $\M$ in $U_\alpha(\rats)/[U_\alpha,U_\alpha](\rats)$. Then, $\M$ is either in the orbit $G(\rats) \cdot \M_{\Oh_\text{min}}$ and $F_{\psi_{U_\alpha}}$ is completely determined by $F_{\Oh_\text{min}}$ or $\M$ is in $G(\rats) \cdot \M_{\Oh_\text{ntm}}$ and $F_{\psi_{U_\alpha}}$ is completely determined by $F_{\Oh_\text{ntm}}$ where the latter is only possible for $\alpha = \alpha_2$. 
\end{theorem}

\begin{proof}
    Let us start with $\alpha = \alpha_1$ with Fourier coefficient $F_{\alpha_1}(\M; g)$ where
    \begin{equation}
        U_{\alpha_1}(\ads) = \left\{ 
        \begin{psmallmatrix}
            1 & u_1 & u_2 & u_3 \\
            & 1 & & \\
            & & 1 & \\
            & & & 1
        \end{psmallmatrix} : u_i \in \ads \right\}  \qquad \text{and} \qquad
        \M = 
        \begin{psmallmatrix}
            1 & m_1 & m_2 & m_3 \\
            & 1 & & \\
            & & 1 & \\
            & & & 1
        \end{psmallmatrix} \, .
    \end{equation}
    Since $\psi_U$ is assumed to be non-trivial at least one of the charges $m_i$ is non-zero. Without loss of generality we can assume that $m_3 = m_3' \neq 0$ since otherwise the non-zero charge can be moved to the position of $m_3$ using a Weyl reflection $w_1$ part of $L_{\alpha_1}(\rats)$.

    Then,
    \begin{equation}
        \begin{split}
            \MoveEqLeft
            F_{\alpha_1}\Big(
            \begin{psmallmatrix}
                1 & m_1 & m_2 & m_3' \\
                & 1 & & \\
                & & 1 & \\
                & & & 1
            \end{psmallmatrix}; w_1 g\Big) \\
            &= F_{\alpha_1}\Big(
            \begin{psmallmatrix}
                1 & & & m_3' \\
                & 1 & & \\
                & & 1 & \\
                & & & 1
            \end{psmallmatrix};
            \begin{psmallmatrix}
                1 & & & \\
                & 1 & & -m_1/m_3' \\
                & & 1 & -m_2/m_3' \\
                & & & 1
            \end{psmallmatrix} w_1 g\Big) \\
            &= \int E(\chi,
            \begin{psmallmatrix}
                1 & u_1 & u_2 & u_3 \\
                & 1 & & \\
                & & 1 & \\
                & & & 1
            \end{psmallmatrix}
            \begin{psmallmatrix}
                1 & & & \\
                & 1 & & -m_1/m_3' \\
                & & 1 & -m_2/m_3' \\
                & & & 1
            \end{psmallmatrix} w_1 g)
            \overline{\e(m_3' u_3)} \, d^3 u \\
            &= \int F_{(21^2)}(\chi, m_3';
            \begin{psmallmatrix}
                1 & u_1 & u_2 & \\
                & 1 & & \\
                & & 1 & \\
                & & & 1
            \end{psmallmatrix}
            \begin{psmallmatrix}
                1 & & & \\
                & 1 & & -m_1/m_3' \\
                & & 1 & -m_2/m_3' \\
                & & & 1
            \end{psmallmatrix} w_1 g) \, d^2 u
        \end{split}
    \end{equation}
    
    The fact that $\M$ could be conjugated to the form in the second line above with elements in $L_{\alpha_1}(\rats)$ tells us that $\M$ is in the orbit $G(\rats) \cdot \M_{\Oh_\text{min}}$.

    Similarly for $\alpha = \alpha_3$, we have a Fourier coefficient $F_{\alpha_3}(\M; g)$ with
    \begin{equation}
        U_{\alpha_3}(\ads) = \left\{
            \begin{psmallmatrix}
                1 & & & u_3 \\
                & 1 & & u_2 \\
                & & 1 & u_1 \\
                & & & 1
            \end{psmallmatrix} : u_i \in \ads \right\}
        \qquad \text{and} \qquad \M =
        \begin{psmallmatrix}
            1 & & & m_3 \\
            & 1 & & m_2 \\
            & & 1 & m_1 \\
            & & & 1
        \end{psmallmatrix}
    \end{equation}
    where, again, without loss of generality we may assume that $m_3 = m_3' \neq 0$ using $w_3$ in $L_{\alpha_3}(\rats)$. Thus, 
    \begin{equation}
        \begin{split}
            \MoveEqLeft
            F_{\alpha_3}\Big(
            \begin{psmallmatrix}
                1 & & & m_3' \\
                & 1 & & m_2 \\
                & & 1 & m_1 \\
                & & & 1
            \end{psmallmatrix}; w_3 g \Big) \\
            &= F_{\alpha_3}\Big(
            \begin{psmallmatrix}
                1 & & & m_3' \\
                & 1 & & \\
                & & 1 & \\
                & & & 1 
            \end{psmallmatrix};
            \begin{psmallmatrix}
                1 & m_2/m_3' & m_1/m_3' &  \\
                & 1 & & \\
                & & 1 & \\
                & & & 1
            \end{psmallmatrix} w_3 g \Big) \\
            &= \int E(\chi, 
            \begin{psmallmatrix}
                1 & & & u_3 \\
                & 1 & & u_2 \\
                & & 1 & u_1 \\
                & & & 1
            \end{psmallmatrix}
            \begin{psmallmatrix}
                1 & m_2/m_3' & m_1/m_3' &  \\
                & 1 & & \\
                & & 1 & \\
                & & & 1
            \end{psmallmatrix} w_3 g ) \overline{\e(m_3' u_3)} \, d^3 u \\
            &= \int F_{(21^2)}(\chi, m_3';
            \begin{psmallmatrix}
                1 & & & \\
                & 1 & & u_2 \\
                & & 1 & u_1 \\
                & & & 1 
            \end{psmallmatrix}
            \begin{psmallmatrix}
                1 & m_2/m_3' & m_1/m_3' &  \\
                & 1 & & \\
                & & 1 & \\
                & & & 1
            \end{psmallmatrix} w_3 g ) \, d^2 u
        \end{split}
    \end{equation}
    We see that $\M$ is in the orbit $G(\rats) \cdot \M_{\Oh_\text{min}}$.

    Finally, for $\alpha = \alpha_2$ we have a Fourier coefficient $F_{\alpha_2}(\M; g)$ with
    \begin{equation}
        U_{\alpha_2}(\ads) = \left\{
            \begin{psmallmatrix}
                1 & & u_1 & u_2 \\
                & 1 & u_3 & u_4 \\
                & & 1 & \\
                & & & 1
            \end{psmallmatrix} \right\} \qquad \text{and} \qquad
            \M =
            \begin{psmallmatrix}
                1 & & m_1 & m_2 \\
                & 1 & m_3 & m_4 \\
                & & 1 & \\
                & & & 1
            \end{psmallmatrix}
    \end{equation}
    
    If $m_1 m_4 - m_2 m_3 \neq 0$ the coefficient $F_{\alpha_2}$ is attached to $\Oh_\text{ntm}$ according to table \ref{tab:SL4-orbit-coefficients}, see also example~\ref{SL4NTM}. $\M$ is then trivially in $G(\rats) \cdot \M_{\Oh_\text{ntm}}$. 

    Let now $m_1 m_4 - m_2 m_3 = 0$. Since we have assumed that $\psi_U$ is non-trivial at least one of the charges is non-zero and using a Weyl reflection $w_2$ in $L_{\alpha_2}(\rats)$ we can assume $m_2 = m_2' \neq 0$.
    
    \begin{equation}
        \begin{split}
            \MoveEqLeft
            F_{\alpha_2}\Big(
            \begin{psmallmatrix}
                1 & & m_1 & m_2' \\
                & 1 & m_3 & m_4 \\
                & & 1 & \\
                & & & 1
            \end{psmallmatrix}; w_2 g\Big) \\
            &=
            F_{\alpha_2}\Big(
            \begin{psmallmatrix}
                1 & & & m_2' \\
                & 1 & m_3 - \frac{m_1 m_4}{m_2'} & \\
                & & 1 & \\
                & & & 1
            \end{psmallmatrix};
            \begin{psmallmatrix}
                1 & m_4/m_2' & & \\
                & 1 & & \\
                & & 1 & -m_1/m_2'\\
                & & & 1
            \end{psmallmatrix} w_2 g \Big) \\
            &=
            F_{\alpha_2}\Big(
            \begin{psmallmatrix}
                1 & & & m_2' \\
                & 1 & & \\
                & & 1 & \\
                & & & 1
            \end{psmallmatrix};
            \begin{psmallmatrix}
                1 & m_4/m_2' & & \\
                & 1 & & \\
                & & 1 & -m_1/m_2'\\
                & & & 1
            \end{psmallmatrix} w_2 g \Big) \\
            &= \int E(\chi,
            \begin{psmallmatrix}
                1 & & u_1 & u_2 \\
                & 1 & u_3 & u_4 \\
                & & 1 & \\
                & & & 1
            \end{psmallmatrix}
            \begin{psmallmatrix}
                1 & m_4/m_2' & & \\
                & 1 & & \\
                & & 1 & -m_1/m_2'\\
                & & & 1
            \end{psmallmatrix} w_2 g) \overline{\e(m_2' u_2)} \, d^4 u \\
            &= \int F_{(21^2)}(\chi, m_2'; 
            \begin{psmallmatrix}
                1 & & u_1 &  \\
                & 1 & u_3 & u_4 \\
                & & 1 & \\
                & & & 1
            \end{psmallmatrix}
            \begin{psmallmatrix}
                1 & m_4/m_2' & & \\
                & 1 & & \\
                & & 1 & -m_1/m_2'\\
                & & & 1
            \end{psmallmatrix} w_2 g) \, d^3 u
        \end{split}
    \end{equation}
    Again, $\M$ is seen to be in the orbit $G(\rats) \cdot \M_{\Oh_\text{min}}$.
\end{proof}

Let us draw some parallells to the proof of the main theorem in Miller-Sahi \cite{MillerSahi} discussed at the end of section \ref{sec:deg-Whittaker}. The statement that $\psi_U$ is $L$-conjugate to $\psi|_U$ as defined in \eqref{characterrestrict} is parallel to the statement that $\M$ is in the orbit $G(\rats) \cdot \M_{\Oh_\text{min}}$ after noting that all conjugations in the proof of theorem \ref{thm:SL4-maximal-parabolic} were done with elements in $L_\alpha(\rats)$. The statement for $\M$ is generalizable to the next-to-minimal representation and the orbit $G(\rats) \cdot \M_{\Oh_\text{ntm}}$. 

Using theorem \ref{thm:min-rep} and appendix \ref{app:SL4-F211-alternative-expansion} we find that $F_{\Oh_\text{min}}$ can be expanded in terms of maximally degenerate Whittaker vectors all charged on a single simple root $\alpha$ which can be chosen to be either $\alpha_1$, $\alpha_2$ or $\alpha_3$ and thus to correspond to the root defined by $\psi|_U$ and $\psi_\alpha$ in \eqref{characterrestrict}. This means that, in the minimal representation, Fourier coefficients on maximal parabolic subgroups defined by a root $\alpha$ can be expressed in terms of Whittaker vectors on $N$ with support only on $\alpha$.

We expect theorem \ref{thm:SL4-maximal-parabolic} to generalise for general Fourier coefficients beyond the case of maximal parabolic subgroups, that is, if $\M$ is in the orbit $\Oh$, then the Fourier coefficient $F(\M; g)$ associated to $\M$ should be linearly determined by the orbit coefficient $F_\Oh$.

Before concluding this section, let us also briefly discuss the next-to-next-to-minimal representation associated with the orbit $\Oh = (31)$. From table \ref{tab:SL4-Whittaker-as-orbits} we expect that the $F_{(31)}$ coefficient can be expanded in terms of Whittaker vectors $W_N(\chi, m_1', m_4', 0; g)$ and $W_N(\chi, 0, m_4', m_6'; g)$ together with $W_N(\chi, m_1', m_4', m_6'; g)$ attached to the orbit $(4)$ similar to $F_{(2^2)}$ in lemma \ref{lem:SL4-22-expanded} and $F_{(21^2)}$ in lemma \ref{lem:SL4-211-expanded}. This would mean that, in the next-to-next-to-minimal representation, $F_{(31)}$ would be expressed in terms of Whittaker vectors on $N$ charged on two non-commuting roots (type $A_2$).

\section{Spherical vectors for minimal representations of exceptional groups}

So far in this paper all our results have been global, i.e. pertaining to automorphic representations $\pi$ of adelic groups $G(\mathbb{A})$. In cases where $\pi$ factorises according to 
\beq 
\pi=\pi_{\infty} \otimes \bigotimes_{p<\infty} \pi_p, 
\eeq
one can sometimes obtain local results for spherical vectors $f_p^\circ$ of $\pi_p$ by factoring global vectors in $\pi$. In this section we shall use this fact to show that our results can be used to calculate local spherical Whittaker vectors associated with minimal automorphic representations of the exceptional groups $E_6, E_7, E_8$. In particular, we shall prove Propositions \ref{prop:E6}, \ref{prop:E7} and \ref{prop:E8} stated in the introduction. These results constitute strong support for our assertion that  theorems \ref{thm:orbit-exp}, \ref{thm:min-rep}, \ref{thm:ntm-rep} can be extended to all simply-laced, simple Lie groups. 

\subsection{Spherical vectors}
Consider the local principal series $\text{Ind}_{B(\mathbb{Q}_p)}^{G(\mathbb{Q}_p)}\chi_p$ for $G(\mathbb{Q}_p)$ a split, reductive algebraic group over $\mathbb{Q}_p$ and $B=NA$ its Borel subgroup. At the infinite place $p=\infty$, $G(\mathbb{Q}_\infty)=G(\mathbb{R})$ is a split real Lie group. Up to scalar multiple there is a unique spherical standard section $f_p^{\circ}\in \text{Ind}_{B(\mathbb{Q}_p)}^{G(\mathbb{Q}_p)}\chi_p$ defined by 
\beq
f^{\circ}_p(nak) =\chi_p(na)=\chi_p(a).
\eeq
The module $\text{Ind}_{B(\mathbb{Q}_p)}^{G(\mathbb{Q}_p)}\chi_p$ has an embedding
\beq
  \text{Ind}_{B(\mathbb{Q}_p)}^{G(\mathbb{Q}_p)}\chi_p \hookrightarrow \text{Ind}_{N(\mathbb{Q})}^{G(\mathbb{Q}_p)} \psi_p, 
\eeq
where $\psi_p$ is a unitary character on the unipotent radical $N(\mathbb{Q}_p)$. The image of the principal series representation inside $\text{Ind}_{N(\mathbb{Q})}^{G(\mathbb{Q}_p)} \psi_p$ is called a \emph{Whittaker model}, denoted by $Wh_{\psi_p}$, and its elements are \emph{Whittaker vectors}. Associated with $f_p^{\circ}$ there is a unique \emph{spherical Whittaker vector} defined by
\beq
\label{sphWh}
W_{\psi_p}^{\circ}(\chi_p, g):=\int_{N(\mathbb{Q}_p)} \chi_p(w_0 ng)\overline{\psi_p(n)}dn,
\eeq
satisfying
\beq
W^{\circ}_{\psi_p}(\chi_p, ngk)=\psi_p(n)W^{\circ}_{\psi_p}(\chi_p, g).
\eeq
Formula~\eqref{sphWh} is the local analogue of~\eqref{Wgeneric} for generic characters $\psi_p$. When $\psi_p$ is degenerate, one has to start from the definition~\eqref{Wvector} and evaluate $W^\circ_{\psi_p}(\chi_p,g)$ using the formalism developed in~\cite{FKP}.

For finite places $p<\infty$ there is a general formula for $W_{\psi_p}^{\circ}$ for generic $\psi_p$ due to Shintani \cite{Shintani} and Casselman-Shalika \cite{CasselmanShalika}, famously given in terms of characters of the Langlands dual group ${}^LG(\mathbb{C})$. In the special case of $SL(2,\mathbb{Q}_p)$ this formula reads
\beq
W^{\circ}_{\psi_p}(s, v)=|v|_p^{-2s+2}\gamma_p(mv^2)(1-p^{-2s})\frac{1-p^{-2s+1}|mv^2|_p^{2s-1}}{1-p^{-2s+1}},
\eeq
with $m\in \mathbb{Q}_p^{\times}$ parametrising the character $\psi_p$  and we parametrized the torus $A(\mathbb{Q}_p)$ by 
\beq
\begin{pmatrix} v & \\ & v^{-1} \end{pmatrix}, \qquad v\in \mathbb{Q}_p^*.
\eeq
At the real place $p=\infty $ no such general formula exists for arbitrary reductive groups, although for $SL(n, \mathbb{R})$, Stade \cite{MR1047756} has found a formula in terms of nested integrals over products of modified Bessel functions (generalising a famous formula of Bump and Vinogradov-Takhtajan for $SL(3,\mathbb{R})$).

\subsection{Minimal representations of exceptional groups}
For induced representations $\text{Ind}_{P(\mathbb{Q}_p)}^{G(\mathbb{Q}_p)}\chi_{P,p}$ associated with other standard parabolic subgroups $P=LU$ much less is known about the associated spherical vectors. In general one does not expect a multiplicity one theorem in this case, i.e. the space of spherical vectors inside $\text{Ind}_{U(\mathbb{Q}_p)}^{G(\mathbb{Q}_p)} \psi_{U, p}$ might be more than one-dimensional. However, the situation improves for minimal representations, i.e. those with smallest non-trivial Gelfand-Kirillov dimension. It is known that the minimal representation $\pi_{min,p}$ can always be realised as a sub-module inside $\text{Ind}_{P(\mathbb{Q}_p)}^{G(\mathbb{Q}_p)}\chi_{P,p}$ for $P$ a maximal parabolic and $\chi_{P,p}$ chosen suitably. When $P$ is maximal, $\chi_{P, p}$ can be parametrised by a single complex variable $s$ and for some special value $s=s_0$ one has 
\beq
\pi_{min}\subset \text{Ind}_{P(\mathbb{Q}_p)}^{G(\mathbb{Q}_p)}\chi_{P, p}(s)\Big|_{s=s_0}\hookrightarrow \text{Ind}_{U(\mathbb{Q}_p)}^{G(\mathbb{Q}_p)} \psi_{U, p}.
\eeq
For the minimal representation one has again a multiplicity one theorem (see \cite{GanSavin}) and so there is a unique spherical vector $F_{U}^{\circ}\in \text{Ind}_{U(\mathbb{Q}_p)}^{G(\mathbb{Q}_p)} \psi_{U, p}$. 

In what follows we shall discuss explicit examples of spherical vectors in different realisations of minimal representations of $E_6, E_7, E_8$. For later use we record the functional dimensions of these minimal representations below:
\beq
\text{GKdim}\, \pi_{min} = \left\{\begin{array}{cc} 
11, & \hspace{.3cm} E_6\\
17, & \hspace{.3cm} E_7 \\
29, & \hspace{.3cm} E_8 \\ 
\end{array}\right.
\eeq
By referring back to example~\ref{E7minorb} we note that the (Gelfand--Kirillov) dimension of the minimal representation is half the (complex) dimension of the minimal nilpotent orbit of $E_7(\cx)$. This is a general feature of the wavefront set and special unipotent representations.

\subsubsection{Abelian realisation}

\paragraph{Finite places ($p<\infty$).} For the exceptional groups $E_6, E_7$ the $p$-adic spherical  vectors in the minimal representation have been found by Savin and Woodbury in \cite{SavinWoodbury}. They use a realisation of $\pi_{min,p}$ embedded in $\text{Ind}_{Q(\mathbb{Q}_p)}^{G(\mathbb{Q}_p)} \psi_{Q, p}$ where $P_Q=LQ$ is a maximal parabolic subgroup such that the unipotent radical $Q$ is \emph{abelian}. 
These unipotent radicals are associated with the following 3-gradings of the underlying Lie algebras 
\beqa
{}\mathfrak{e}_{6}&=&\mathfrak{g}_{-1}\oplus \mathfrak{g}_0\oplus \mathfrak{g}_1\, \, =\, \, {\bf 16}\oplus (\mathfrak{so}(5,5)\oplus {\bf 1}) \oplus {\bf 16}
\nn \\
{}\mathfrak{e}_{7}&=&\mathfrak{g}_{-1}\oplus \mathfrak{g}_0\oplus \mathfrak{g}_1\, \, =\, \, {\bf 27} \oplus (\mathfrak{e}_6\oplus {\bf 1}) \oplus {\bf 27},
\eqa
where the $\mathfrak{g}_1$-space is the Lie algebra of the unipotent radical $Q$. Savin and Woodbury then give the following formula for the minimal spherical vectors:
\beqa
{} \mathfrak{e}_6 &:& F_{Q,p}^{\circ}(x) = \frac{1-p^2|x|_p^{-2}}{1-p^2},
\nn \\
{} \mathfrak{e}_7 &:& F_{Q,p}^{\circ}(x) = \frac{1-p^3|x|_p^{-3}}{1-p^3}.
\label{SWvectors}
\eqa
We have evaluated the spherical vector at the identity $1\in A(\mathbb{Q}_p)$, and viewed $F^{\circ}_{\psi_{Q}, p}$ as a function $F^{\circ}_{Q, p}(x)$ of the ``charge'' $x\in \mathbb{Q}^{\times}$ which characterises $\psi_{Q, p}$ (along one direction in $[Q, Q]\backslash Q=Q$ since $Q$ is abelian).

\paragraph{Real place ($p=\infty$).}  At the real place, the spherical vector in the same realisation of the minimal representation of $E_6, E_7$ was obtained by Dvorsky and Sahi in \cite{DS} with the result:
\beqa
{}\mathfrak{e}_{6}&:&F_{Q, \infty}^{\circ}(x) =x^{-1}K_1(x),
\nn \\
{} \mathfrak{e}_7&:& F_{Q, \infty}^{\circ}(x) =x^{-3/2}K_{3/2}(x),
\label{DSvectors}
\eqa
where $K_t$ denotes the modified Bessel function.

\subsubsection{Heisenberg realisation}
There is another model for the minimal representation where the parabolic $P_U=MU$ is such that the unipotent radical $U$ is non-abelian. This realisation is associated with a 5-grading of the Lie algebra, and therefore also exists for $\mathfrak{e}_8$ (which does not admit a 3-grading). Specifically, we have: 
\beqa
{}\mathfrak{e}_{6}&=&{\bf 1} \oplus {\bf 20}\oplus (\mathfrak{sl}(6,\mathbb{R})\oplus {\bf 1}) \oplus {\bf 20}\oplus {\bf 1}
\nn \\
{}\mathfrak{e}_{7}&=&{\bf 1} \oplus {\bf 32}\oplus (\mathfrak{so}(6,6)\oplus {\bf 1}) \oplus {\bf 32}\oplus {\bf 1},
\nn \\
{}\mathfrak{e}_{8}&=&{\bf 1} \oplus {\bf 56}\oplus (\mathfrak{e}_7\oplus {\bf 1}) \oplus {\bf 56}\oplus {\bf 1}.
\label{5gradings}
\eqa
The Lie algebra of $U$ is the positive part of the grading, and so has the structure of a Heisenberg algebra, thereby explaining the name ``Heisenberg realisation''. In the physics literature, this realisation is called the quasi-conformal realisation~\cite{Gunaydin:2000xr}.

\paragraph{Finite places ($p<\infty$).}  At the finite places the minimal representation $\pi_{min,p}$ is  realised as a submodule of the induced representation $\text{Ind}_{U(\mathbb{Q}_p)}^{G(\mathbb{Q}_p)}\chi_{P_U, p}(s)$ for some value $s=s_0$. The values of $s_0$ can be found in~\cite{GRS}. Since $U$ is a Heisenberg group, there are Whittaker models associated with characters on $U$ which are trivial on the centre $\mathcal{Z}=[U,U]$ as well as one-dimensional characters which are non-trivial on $\mathcal{Z}$ (but trivial on $\mathcal{Z}\backslash U$). This is a consequence of the Stone-von-Neumann theorem. Denoting by $\psi_{\mathcal{Z}, p}$ a unitary character on $\mathcal{Z}$ we thus have two different function spaces in which to embed the minimal representation: 
\beq
\text{Ind}_{U(\mathbb{Q}_p)}^{G(\mathbb{Q}_p)} \psi_{U, p}, \qquad \text{or}\qquad \text{Ind}_{\mathcal{Z}(\mathbb{Q}_p)}^{G(\mathbb{Q}_p)} \psi_{\mathcal{Z}, p}.
\eeq
One can roughly think of the elements of the former as ``abelian spherical vectors'' and elements of the latter as ``non-abelian spherical vectors''. This is due to the fact that the latter will naturally depend on a variable along the centre of the Heisenberg group $\mathcal{Z}$, while the abelian vectors will only be functions on the abelianization $\mathcal{Z}\backslash U$.

Kazhdan and Polishchuk \cite{KazhdanPolishchuk} have constructed non-abelian $p$-adic spherical vectors for $E_6, E_7, E_8$. These spherical vectors depend on a set of rational variables $(y, x_0, x_1, \dots, x_n)\in {\mathbb{Q}}^{n+2}$ where $n=9, 15, 27$ for $E_6, E_7, E_8$, respectively. The variable $y$ parametrizes the centre $\mathfrak{g}_2$ of the Heisenberg algebra, while $(x_0, x_1, \dots, x_n)$ parametrize a Lagrangian subspace of $\mathfrak{g}_1$. We denote these $p$-adic non-abelian spherical vectors by $F^{na}_{U,p}(y; x_0, x_1, \dots, x_n)$. Kazhdan and Polishchuk show that at the locus $x_0\neq0$ one can consistently take an abelian limit $y\to 0$, yielding abelian spherical vectors $F_{U,p}(x_0, x_1, \dots, x_n)$. Restricting further to the locus where all $x_i=0$, $i=1, \dots, n$, one can extract the following form of the spherical vectors from \cite{KazhdanPolishchuk}:
\beqa
{}\mathfrak{e}_{6}&:& F_{U,p}(x_0)=|x_0|_p^{-2} \frac{1-p|x_0|_p^{-1}}{1-p}
\nn \\ 
{} \mathfrak{e}_7&:&F_{U,p}(x_0)=|x_0|_p^{-3} \frac{1-p^{2}|x_0|_p^{-2}}{1-p^2}
\nn \\
{} \mathfrak{e}_8&:&F_{U,p}(x_0)=|x_0|_p^{-5} \frac{1-p^{4}|x_0|_p^{-4}}{1-p^4}.
\label{KPvectors}
\eqa
The further restriction to dependence only on $x_0$ is in line with taking a representative of the $M$-orbit on $U$ on the single simple root space that defines the maximal parabolic $P_U$.

\paragraph{Real place ($p=\infty$).} The real spherical vectors in the Heisenberg realisation of the minimal representations of $E_6, E_7, E_8$ were obtained by Kazhdan-Pioline-Waldron in \cite{Kazhdan:2001nx}. With similar manipulations as above one can extract the abelian limit of these spherical vectors with the result:
\beqa
{}\mathfrak{e}_{6}&:& F_{U,\infty}(x_0)=x_0^{-5/2} K_{1/2}(x_0)
\nn \\ 
{} \mathfrak{e}_7&:&F_{U,\infty}(x_0)=x_0^{-4} K_1(x_0)
\nn \\
{} \mathfrak{e}_8&:&F_{U,\infty}(x_0)=x_0^{-7} K_2(x_0). 
\label{KPWvectors}
\eqa

\subsection{Relation with degenerate Whittaker vectors}
\label{sec_proofprop}
The spherical vectors analyzed in the previous subsections should be viewed as Fourier coefficients with respect to the relevant unipotent radicals. For example, suppose we have an automorphic form $\varphi$ in the minimal representation $\pi_{min}$ of $E_6, E_7, E_8$. We can then consider its non-abelian Fourier expansion with respect to the Heisenberg unipotent radical $U$. Schematically this expansion takes the form
\beqa
\varphi&=& \sum_{\substack{y\in \mathbb{Q}^{\times} \\ (x_0, x_1, \dots, x_n)\in \mathbb{Q}^{n+1}}} \Bigg[\prod_{p<\infty} F^{na}_{U,p}(y;x_0, x_1, \dots, x_n)\Bigg] F^{na}_{U, \infty}(y;x_0, x_1, \dots, x_n) 
\nn \\
& & + \sum_{(x_0, x_1, \dots x_n)\in \mathbb{Q}^{n+1}} \Bigg[\prod_{p<\infty}F_{U,p}(x_0, x_1, \dots, x_n)\Bigg] F_{U, \infty}(x_0, x_1, \dots, x_n),
\eqa
where we have suppressed the constant terms. The coefficients in this expansion correspond precisely to the spherical vectors in the Heisenberg realisation of the minimal representation. 

Similarly, if we expand $\varphi$ along the maximal abelian unipotent radical $Q$ (now restricting to $E_6, E_7$) the expansion instead takes the form 
\beq
\varphi=\sum_{(x_0, x_1, \dots, x_r)\in \mathbb{Q}^{r+1}} \Bigg[\prod_{p<\infty}F_{Q, p}(x_0, x_1, \dots, x_r)\Bigg]F_{Q, \infty}(x_0, x_1, \dots, x_r),
\eeq
where $r=10, 16$ for $E_6, E_7,$ respectively. The Fourier coefficients are in this case purely abelian and correspond to the spherical vectors in the Savin-Woodbury realisation of the minimal representation.

Following the arguments of section \ref{sec:deg-Whittaker} we expect that the abelian spherical vectors in the minimal representation should be captured by certain maximally degenerate Whittaker vectors along the maximal unipotent radical $N$ (associated with the Borel subgroup).

As mentioned in the introduction, one can realise the minimal representation of $G=E_6,E_7,E_8$ using the maximal parabolic Eisenstein series associated with node $1$ of the Dynkin diagram in Bourbaki labelling, see~\cite{GreenSmallRep} (following \cite{GRS}). There is a one parameter family $E(s,g)$ of such Eisenstein series and for $s=3/2$ the series belongs to the minimal representation.  All maximally degenerate Whittaker vectors of $E(3/2, g)$ were computed in \cite{FKP} and we now wish to compare those results with the minimal spherical vectors discussed above. 

\subsubsection{Proof of Proposition 1.1}  

In this section we consider the case $G=E_6$. Let us begin with the abelian case corresponding to the Savin-Woodbury realisation of the minimal representation. The unipotent radical $Q$ contains a single simple root which is $\alpha_1$ in Bourbaki labelling. We therefore expect the associated spherical vector to arise from a degenerate Whittaker vector on $N$ which is non-trivial only along this simple root. This is the Whittaker vector labelled by the vector $(m, 0, 0, 0, 0, 0)$ in Table A.1 of \cite{FKP}. Here, $m\in \mathbb{Q}$ parametrises the maximally degenerate unitary character $\psi_{\alpha_1}$ on $N$. From Table A.1 of~\cite{FKP} we then find the following expression for the degenerate Whittaker vector along $\alpha_1$:
\beq
W_{\psi_{\alpha_1}}(3/2, 1)=\frac{2}{\xi(3)}|m|^{-1} \sigma_2(m) K_1(m).
\eeq
As is well-known, the divisor sum $\sigma_s(m)$ can be written as an Euler product  for integer $m$ and $s\in\mathbb{C}$ according to
\begin{align}
\sigma_s(m) = \prod_{p<\infty} \frac{1-p^s |m|_p^{-s}}{1-p^s}.
\end{align}
Inserting this into the expression for the Whittaker vector, we get
\beq
W_{\psi_{\alpha_1}}(3/2, 1)=\frac{2}{\xi(3)} \left[\prod_{p<\infty}\frac{1-p^2|m|_p^{-2}}{1-p^2}\right] |m|^{-1}K_1(m),
\eeq
where we recognise Savin and Woodbury's p-adic spherical vector $F_{Q,p}^{\circ}(m)$ inside the brackets and the Dvorsky-Sahi real spherical vector $F_{Q, \infty}^{\circ}(m)$ outside the brackets. We conclude that in this case the global degenerate Whittaker vector indeed reproduces the spherical vector in the minimal representation at all local places.

Let us now also study the Heisenberg realisation. In this case the unipotent radical $U$ contains the simple root $\alpha_2$. Degenerate Whittaker vectors non-trivial along this root correspond to $(0, m, 0, 0, 0, 0)$ in Table A.1 of \cite{FKP} and we find 
\begin{align}
W_{\psi_{\alpha_2}}(3/2, 1)&=\frac{2}{\xi(3)} |m|^{-1/2} \sigma_1(m) K_{1/2}(m)\nn\\
&=\frac{2}{\xi(3)} \left[\prod_{p<\infty} |m|_p^{-2} \frac{1-p|m|_p^{-1}}{1-p}\right] |m|^{-5/2} K_{1/2}(m),
\end{align}
where in the second equality we used the fact that $\prod_{p<\infty} |m|_p^{-2} =|m|^2$. Comparing this with eqs. (\ref{KPvectors}) and (\ref{KPWvectors}) we see that $p$-adic and real spherical vectors are correctly reproduced. {\flushright{$\square$}\\}

\subsubsection{Proof of Proposition 1.2}

Let now $G=E_7$. The abelian unipotent radical $Q$ of $E_7$ arises from the decomposition with respect to node $7$ in the Dynkin diagram. Hence the degenerate Whittaker vector which is relevant in this case is $W_{\psi_{\alpha_7}}$; this is the case $(0, 0, 0, 0, 0, 0, m)$ in Table A.2 of \cite{FKP}, yielding
\beq
W_{\psi_{\alpha_7}}(3/2, 1)=\frac{2}{\xi(4)} |m|^{-3/2} \sigma_3(m) K_{3/2}(m)=\frac{2}{\xi(4)} \left[\prod_{p<\infty} \frac{1-p^3|m|_p^{-3}}{1-p^3}\right] |m|^{-3/2} K_{3/2}(m).
\eeq
Inside the bracket we now recognise the Savin-Woodbury p-adic spherical vector $F_{Q,p}^{\circ}(m)$ in (\ref{SWvectors}) and outside the brackets we recognise the Dvorsky-Sahi real spherical vector $F_{Q,\infty}^{\circ}(m)$ in (\ref{DSvectors}). Thus, also for the abelian realisation of the minimal representation of $E_7$ everything is correctly reproduced. 

Now, let us turn to the Heisenberg realisation. The Heisenberg unipotent radical $U$ of $E_7$ corresponds to the decomposition with respect to node 1, and hence we should compare with the degenerate Whittaker vector $(m, 0, 0, 0, 0, 0, 0)$ in Table A.2 of \cite{FKP}. From there we find 
\beq
W_{\psi_{\alpha_1}}(3/2, 1)= \frac{2}{\xi(3)} |m|^{-1} \sigma_2(m) K_1(m)=\frac{2}{\xi(3)} \left[\prod_{p<\infty} |m|_p^{-3} \frac{1-p^2|m|_p^{-2}}{1-p^2}\right]|m|^{-4} K_1(m).
\eeq
Again we see that the spherical vectors in (\ref{KPvectors}) and (\ref{KPWvectors}) are correctly reproduced. {\flushright{$\square$}\\}

\subsubsection{Proof of Proposition 1.3}

Finally we consider the case $G=E_8$. The group $E_8$ does not exhibit a 3-grading so the only realisation of the minimal representation is the one associated with the 5-grading (\ref{5gradings}). The unipotent radical $U$ is a 57-dimensional Heisenberg group which contains the simple root $\alpha_8$. Thus the relevant degenerate Whittaker vector is $(0, 0, 0, 0, 0, 0, 0, m)$ in Table A.3 of \cite{FKP}. This yields 
\begin{align}
W_{\psi_{\alpha_8}}(3/2, 1)&= \frac{2 \xi(4)}{\xi(3)\xi(5)} |m|^{-2}\sigma_4(m)K_2(m)\nn\\
&=\frac{2 \xi(4)}{\xi(3)\xi(5)}\left[\prod_{p<\infty} |m|_p^{-5}\frac{1-p^4|m|_p^{-4}}{1-p^4}\right] |m|^{-7}K_2(m).
\end{align}
As before this matches nicely with (\ref{KPvectors}) and (\ref{KPWvectors}).{\flushright{$\square$}\\}

\pagebreak

\appendix

\section{\texorpdfstring{$SL(4)$}{SL(4)} Whittaker vectors in terms of orbit coefficients}
\label{app:SL4-Whittaker-as-orbits}

The constant term and the generic Whittaker vectors on $N$ are directly attached to the trivial and regular orbits respectively.

\paragraph{Maximally degenerate Whittaker vectors on $N$}

\begin{equation}
    \begin{split}
        \MoveEqLeft 
        \int F_{(21^2)}(\chi, m_1'; 
        \begin{psmallmatrix}
            1 & & & \\
            & & & -1 \\
            & & 1 & \\
            & 1 & &
        \end{psmallmatrix}
        \begin{psmallmatrix}
            1 & & u_2 & u_3 \\
            & 1 & u_4 & u_5 \\
            & & 1 & u_6 \\
            & & & 1
        \end{psmallmatrix}
        g) \, d^5u \\
        &= \int E(\chi,
        \begin{psmallmatrix}
            1 & x_1 & & \\
            & 1 & & \\
            & & 1 & \\
            & & & 1
        \end{psmallmatrix}
        \begin{psmallmatrix}
            1 & & u_2 & u_3 \\
            & 1 & u_4 & u_5 \\
            & & 1 & u_6 \\
            & & & 1
        \end{psmallmatrix}
        g) \overline{\e(m_1' x_1)} \, dx_1 \, d^5u \\
        &= \int E(\chi,
        \begin{psmallmatrix}
            1 & x_1 & x_2 & x_3 \\
            & 1 & x_4 & x_5 \\
            & & 1 & x_6 \\
            & & & 1
        \end{psmallmatrix}
        g) \overline{\e(m_1' x_1)} \, d^6x = W_N(\chi, m_1', 0, 0; g)
    \end{split}
\end{equation}
where $x_2 = u_2 + u_4 x_1$, $x_3 = u_3 + u_5 x_1$ and the rest $x_i = u_i$.

\begin{equation}
    \begin{split}
        \MoveEqLeft 
        \int F_{(21^2)}(\chi, m_4'; 
        \begin{psmallmatrix}
            & 1 & & \\
            1 & & & \\
            & & & 1 \\
            & & 1 &
        \end{psmallmatrix}
        \begin{psmallmatrix}
            1 & u_1 & u_2 & u_3 \\
            & 1 & & u_5 \\
            & & 1 & u_6 \\
            & & & 1
        \end{psmallmatrix}
        g) \, d^5u \\
        &= \int E(\chi,
        \begin{psmallmatrix}
            1 & & & \\
            & 1 & x_4 & \\
            & & 1 & \\
            & & & 1
        \end{psmallmatrix}
        \begin{psmallmatrix}
            1 & u_1 & u_2 & u_3 \\
            & 1 & & u_5 \\
            & & 1 & u_6 \\
            & & & 1
        \end{psmallmatrix}
        g) \overline{\e(m_4' x_4)} \, dx_4 \, d^5u \\
        &= \int E(\chi,
        \begin{psmallmatrix}
            1 & x_1 & x_2 & x_3 \\
            & 1 & x_4 & x_5 \\
            & & 1 & x_6 \\
            & & & 1
        \end{psmallmatrix}
        g) \overline{\e(m_4' x_4)} \, d^6x = W_N(\chi, 0, m_4', 0; g) 
    \end{split}
\end{equation}
where $x_5 = u_5 + u_6 x_4$ and the rest $x_i = u_i$.

\begin{equation}
    \begin{split}
        \MoveEqLeft 
        \int F_{(21^2)}(\chi, m_6'; 
        \begin{psmallmatrix}
            & & 1 & \\
            & 1 & & \\
            -1 & & & \\
            & & & 1
        \end{psmallmatrix}
        \begin{psmallmatrix}
            1 & u_1 & u_2 & u_3 \\
            & 1 & u_4 & u_5 \\
            & & 1 & \\
            & & & 1
        \end{psmallmatrix}
        g) \, d^5u \\
        &= \int E(\chi,
        \begin{psmallmatrix}
            1 & & & \\
            & 1 & & \\
            & & 1 & x_6 \\
            & & & 1
        \end{psmallmatrix}
        \begin{psmallmatrix}
            1 & u_1 & u_2 & u_3 \\
            & 1 & u_4 & u_5 \\
            & & 1 & \\
            & & & 1
        \end{psmallmatrix}
        g) \overline{\e(m_6' x_6)} \, dx_6 \, d^5u \\
        &= \int E(\chi,
        \begin{psmallmatrix}
            1 & x_1 & x_2 & x_3 \\
            & 1 & x_4 & x_5 \\
            & & 1 & x_6 \\
            & & & 1
        \end{psmallmatrix}
        g) \overline{\e(m_6' x_6)} \, d^6x = W_N(\chi, 0, 0, m_6'; g) 
    \end{split}
\end{equation}
where $x_i = u_i$.

\pagebreak

\paragraph{Remaining degenerate Whittaker vectors on $N$}

\begin{equation}
    \begin{split}
        \MoveEqLeft
        \sum_m \int F_{(31)}(\chi, m_1', 0, m_4', m;
        \begin{psmallmatrix}
            1 & & & \\
            & 1 & & \\
            & & & -1 \\
            & & 1 &
        \end{psmallmatrix}
        \begin{psmallmatrix}
            1 & & & \\
            & 1 & & u_5 \\
            & & 1 & u_6 \\
            & & & 1
        \end{psmallmatrix}
        g) \, d^2 u \\
        &= \sum_m \int E(\chi, 
        \begin{psmallmatrix}
            1 & x_1 & x_3 & x_2 \\
            & 1 & & x_4 \\
            & & 1 & x_0 \\
            & & & 1
        \end{psmallmatrix} 
        \begin{psmallmatrix}
            1 & & & \\
            & 1 & & \\
            & & & -1 \\
            & & 1 &
        \end{psmallmatrix}
        \begin{psmallmatrix}
            1 & & & \\
            & 1 & & u_5 \\
            & & 1 & u_6 \\
            & & & 1
        \end{psmallmatrix}
        g) \overline{\e(m_1' x_1 + m_4' x_4 + m x_0)} \, d^2 u \, d^5x \\
        &= \int E(\chi, 
        \begin{psmallmatrix}
            1 & x_1 & x_3 & x_2 \\
            & 1 & & x_4 \\
            & & 1 & \\
            & & & 1
        \end{psmallmatrix} 
        \begin{psmallmatrix}
            1 & & & \\
            & 1 & & \\
            & & & -1 \\
            & & 1 &
        \end{psmallmatrix}
        \begin{psmallmatrix}
            1 & & & \\
            & 1 & & u_5 \\
            & & 1 & u_6 \\
            & & & 1
        \end{psmallmatrix}
        g) \overline{\e(m_1' x_1 + m_4' x_4)} \, d^2 u \, d^4x \\
        &= \int E(\chi, 
        \begin{psmallmatrix}
            1 & x_1 & x_2 & -x_3 \\
            & 1 & x_4 & \\
            & & 1 & \\
            & & & 1
        \end{psmallmatrix} 
        \begin{psmallmatrix}
            1 & & & \\
            & 1 & & u_5 \\
            & & 1 & u_6 \\
            & & & 1
        \end{psmallmatrix}
        g) \overline{\e(m_1' x_1 + m_4' x_4)} \, d^2 u \, d^4x \\
        &= \int E(\chi, 
        \begin{psmallmatrix}
            1 & & & -x_3 \\
            & 1 & & \\
            & & 1 & \\
            & & & 1
        \end{psmallmatrix}
        \begin{psmallmatrix}
            1 & x_1 & x_2 & \\
            & 1 & x_4 & \\
            & & 1 & \\
            & & & 1
        \end{psmallmatrix} 
        \begin{psmallmatrix}
            1 & & & \\
            & 1 & & u_5 \\
            & & 1 & u_6 \\
            & & & 1
        \end{psmallmatrix}
        g) \overline{\e(m_1' x_1 + m_4' x_4)} \, d^2 u \, d^4x  
    \end{split}
\end{equation}

Since the integrand is invariant under discrete shifts in $x_3$ we can make the substitution $x_3 \to -x_3$ without introducing an overall minus sign if we keep the integral domain $\rats \bs \ads$ unchanged. Thus, the above expression becomes
\begin{equation}
    \begin{split}
        \MoveEqLeft
        \int E(\chi, 
        \begin{psmallmatrix}
            1 & & & x_3 \\
            & 1 & & \\
            & & 1 & \\
            & & & 1
        \end{psmallmatrix}
        \begin{psmallmatrix}
            1 & x_1 & x_2 & \\
            & 1 & x_4 & \\
            & & 1 & \\
            & & & 1
        \end{psmallmatrix} 
        \begin{psmallmatrix}
            1 & & & \\
            & 1 & & u_5 \\
            & & 1 & u_6 \\
            & & & 1
        \end{psmallmatrix}
        g) \overline{\e(m_1' x_1 + m_4' x_4)} \, d^2 u \, d^4x \\
        &= \int E(\chi, 
        \begin{psmallmatrix}
            1 & x_1 & x_2 & x_3 \\
            & 1 & x_4 & x_5\\
            & & 1 & x_6\\
            & & & 1
        \end{psmallmatrix} 
        g) \overline{\e(m_1' x_1 + m_4' x_4)} \, d^2 u \, d^4x = W_N(\chi, m_1', m_4', 0; g)
    \end{split}
\end{equation}
where we have made the substitutions $x_3 + u_5 x_1 + u_6 x_2 \to x_3$, $u_5 + u_6 x_4 \to x_5$ and $u_6 \to x_6$.
\begin{equation}
    \begin{split}
        \MoveEqLeft
        \sum_m \int F_{(31)}(\chi, m, m_4', 0, m_6';
        \begin{psmallmatrix}
            & 1 & & \\
            -1 & & & \\
            & & 1 & \\
            & & & 1
        \end{psmallmatrix}
        \begin{psmallmatrix}
            1 & u_1 & u_2 & \\
            & 1 & & \\
            & & 1 & \\
            & & & 1
        \end{psmallmatrix}
        g) \, d^2 u \\
        &= \sum_m \int E(\chi, 
        \begin{psmallmatrix}
            1 & x_0 & x_4 & x_5 \\
            & 1 & & x_3 \\
            & & 1 & x_6 \\
            & & & 1
        \end{psmallmatrix}
        \begin{psmallmatrix}
            & 1 & & \\
            -1 & & & \\
            & & 1 & \\
            & & & 1
        \end{psmallmatrix}
        \begin{psmallmatrix}
            1 & u_1 & u_2 & \\
            & 1 & & \\
            & & 1 & \\
            & & & 1
        \end{psmallmatrix}
        g) \overline{\e(m x_0 + m_4' x_4 + m_6' x_6)} \, d^2u \, d^5x \\
        &= \int E(\chi, 
        \begin{psmallmatrix}
            1 &  & x_4 & x_5 \\
            & 1 & & x_3 \\
            & & 1 & x_6 \\
            & & & 1
        \end{psmallmatrix}
        \begin{psmallmatrix}
            & 1 & & \\
            -1 & & & \\
            & & 1 & \\
            & & & 1
        \end{psmallmatrix}
        \begin{psmallmatrix}
            1 & u_1 & u_2 & \\
            & 1 & & \\
            & & 1 & \\
            & & & 1
        \end{psmallmatrix}
        g) \overline{\e(m_4' x_4 + m_6' x_6)} \, d^2u \, d^4x \\
        &= \int E(\chi, 
        \begin{psmallmatrix}
            1 & & & -x_3 \\
            & 1 & x_4 & x_5 \\
            & & 1 & x_6 \\
            & & & 1
        \end{psmallmatrix}
        \begin{psmallmatrix}
            1 & u_1 & u_2 & \\
            & 1 & & \\
            & & 1 & \\
            & & & 1
        \end{psmallmatrix}
        g) \overline{\e(m_4' x_4 + m_6' x_6)} \, d^2u \, d^4x \\
        &= \int E(\chi, 
        \begin{psmallmatrix}
            1 & & & x_3 \\
            & 1 & x_4 & x_5 \\
            & & 1 & x_6 \\
            & & & 1
        \end{psmallmatrix}
        \begin{psmallmatrix}
            1 & u_1 & u_2 & \\
            & 1 & & \\
            & & 1 & \\
            & & & 1
        \end{psmallmatrix}
        g) \overline{\e(m_4' x_4 + m_6' x_6)} \, d^2u \, d^4x \\
        &= \int E(\chi, 
        \begin{psmallmatrix}
            1 & x_1 & x_2 & x_3 \\
            & 1 & x_4 & x_5 \\
            & & 1 & x_6 \\
            & & & 1
        \end{psmallmatrix}
        g) \overline{\e(m_4' x_4 + m_6' x_6)} \, d^2u \, d^4x = W_N(\chi, 0, m_4', m_6'; g)
    \end{split}
\end{equation}
where $x_1 = u_1$ and $x_2 = u_2$.

\begin{equation}
    \begin{split}
        \MoveEqLeft
        \sum_m \int F_{(2^2)}(\chi, -m_1', 0, m, m_6'; 
        \begin{psmallmatrix}
            -1 & & & \\
            & & 1 & \\
            & 1 & & \\
            & & & 1
        \end{psmallmatrix}
        \begin{psmallmatrix}
            1 & & u_2 & \\
            & 1 & u_4 & u_5 \\
            & & 1 & \\
            & & & 1
        \end{psmallmatrix}
        g) \, d^3u \\
        &= \sum_m \int E(\chi,
        \begin{psmallmatrix}
            1 & & x_1 & x_3 \\
            & 1 & x_0 & x_6 \\
            & & 1 & \\
            & & & 1
        \end{psmallmatrix}
        \begin{psmallmatrix}
            -1 & & & \\
            & & 1 & \\
            & 1 & & \\
            & & & 1
        \end{psmallmatrix}
        \begin{psmallmatrix}
            1 & & u_2 & \\
            & 1 & u_4 & u_5 \\
            & & 1 & \\
            & & & 1
        \end{psmallmatrix}
        g) \overline{\e(m x_0 - m_1' x_1 + m_6' x_6)} \, d^3u \, d^4x \\
        &= \int E(\chi,
        \begin{psmallmatrix}
            1 & & x_1 & x_3 \\
            & 1 & & x_6 \\
            & & 1 & \\
            & & & 1
        \end{psmallmatrix}
        \begin{psmallmatrix}
            -1 & & & \\
            & & 1 & \\
            & 1 & & \\
            & & & 1
        \end{psmallmatrix}
        \begin{psmallmatrix}
            1 & & u_2 & \\
            & 1 & u_4 & u_5 \\
            & & 1 & \\
            & & & 1
        \end{psmallmatrix}
        g) \overline{\e(-m_1' x_1 + m_6' x_6)} \, d^3u \, d^3x \\
        &= \int E(\chi,
        \begin{psmallmatrix}
            1 & -x_1 & & -x_3 \\
            & 1 & & \\
            & & 1 & x_6 \\
            & & & 1
        \end{psmallmatrix}
        \begin{psmallmatrix}
            1 & & u_2 & \\
            & 1 & u_4 & u_5 \\
            & & 1 & \\
            & & & 1
        \end{psmallmatrix}
        g) \overline{\e(-m_1' x_1 + m_6' x_6)} \, d^3u \, d^3x \\
        &= \int E(\chi,
        \begin{psmallmatrix}
            1 & x_1 & & x_3 \\
            & 1 & & \\
            & & 1 & x_6 \\
            & & & 1
        \end{psmallmatrix}
        \begin{psmallmatrix}
            1 & & u_2 & \\
            & 1 & u_4 & u_5 \\
            & & 1 & \\
            & & & 1
        \end{psmallmatrix}
        g) \overline{\e(m_1' x_1 + m_6' x_6)} \, d^3u \, d^3x \\
        &= \int E(\chi,
        \begin{psmallmatrix}
            1 & x_1 & x_2 & x_3 \\
            & 1 & x_4 & x_5 \\
            & & 1 & x_6 \\
            & & & 1
        \end{psmallmatrix}
        g) \overline{\e(m_1' x_1 + m_6' x_6)} \, d^6x = W_N(\chi, m_1', 0, m_6'; g)
    \end{split}
\end{equation}
where we have made the substitution $u_2 + u_4 x_1 \to x_2$, $x_3 + u_5 x_1 \to x_3$ and the rest $u_i \to x_i$.

\paragraph{Whittaker vectors on $N'$}
\begin{equation}
    \begin{split}
        \MoveEqLeft
        \int F_{(21^2)}(\chi, m_1'; 
        \begin{psmallmatrix}
            1 & & & \\
            & 1 & & \\
            & & & -1 \\
            & & 1 &
        \end{psmallmatrix}
        \begin{psmallmatrix}
            1 & & & u_2 \\
            & 1 & & u_3 \\
            & & 1 & \\
            & & & 1
        \end{psmallmatrix}
        g) \, d^2u \\
        &= \int E(\chi, 
        \begin{psmallmatrix}
            1 & & x_1 & \\
            & 1 & & \\
            & & 1 & \\
            & & & 1
        \end{psmallmatrix}
        \begin{psmallmatrix}
            1 & & & u_2 \\
            & 1 & & u_3 \\
            & & 1 & \\
            & & & 1
        \end{psmallmatrix}
        g) \overline{\e(m_1' x_1)} \, d^2u \, dx_1 \\
        &= \int E(\chi, 
        \begin{psmallmatrix}
            1 & & x_1 & x_2 \\
            & 1 & & x_3 \\
            & & 1 & \\
            & & & 1
        \end{psmallmatrix}
        g) \overline{\e(m_1' x_1)} \, d^3x = W_{N'}(\chi, m_1', 0, 0; g)
    \end{split}
\end{equation}

\begin{equation}
    \begin{split}
        \MoveEqLeft
        \int F_{(21^2)}(\chi, m_3'; 
        \begin{psmallmatrix}
            & 1 & & \\
            -1 & & & \\
            & & 1 & \\
            & & & 1
        \end{psmallmatrix}
        \begin{psmallmatrix}
            1 & & u_1 & u_2 \\
            & 1 & & \\
            & & 1 & \\
            & & & 1
        \end{psmallmatrix}
        g) \, d^2u \\
        &= \int E(\chi, 
        \begin{psmallmatrix}
            1 & & & \\
            & 1 & & x_3 \\
            & & 1 & \\
            & & & 1
        \end{psmallmatrix}
        \begin{psmallmatrix}
            1 & & u_1 & u_2 \\
            & 1 & & \\
            & & 1 & \\
            & & & 1
        \end{psmallmatrix}
        g) \overline{\e(m_3' x_3)} \, d^2u \, dx_1 \\
        &= \int E(\chi, 
        \begin{psmallmatrix}
            1 & & x_1 & x_2 \\
            & 1 & & x_3 \\
            & & 1 & \\
            & & & 1
        \end{psmallmatrix}
        g) \overline{\e(m_3' x_3)} \, d^3x = W_{N'}(\chi, 0, 0, m_3'; g) 
    \end{split}
\end{equation}

\begin{equation}
    \begin{split}
        \MoveEqLeft
        \sum_m F_{(2^2)}(\chi, m_1', 0, m, m_3'; g) \\
        &= \sum_m \int E(\chi,
        \begin{psmallmatrix}
            1 & & x_1 & x_2 \\
            & 1 & x_0 & x_3 \\
            & & 1 & \\
            & & & 1
        \end{psmallmatrix}
        g) \overline{\e(m x_0 + m_1' x_1 + m_3' x_3)} \, d^4x \\
        &= \int E(\chi,
        \begin{psmallmatrix}
            1 & & x_1 & x_2 \\
            & 1 & & x_3 \\
            & & 1 & \\
            & & & 1
        \end{psmallmatrix}
        g) \overline{\e(m_1' x_1 + m_3' x_3)} \, d^3x = W_{N'}(\chi, m_1', 0, m_3'; g)
    \end{split}
\end{equation}

\begin{equation}
    \begin{split}
        \MoveEqLeft
        \int F_{(21^2)}(\chi, m_2'; 
        \begin{psmallmatrix}
            1 & & u_1 & \\
            & 1 & & u_3 \\
            & & 1 & \\
            & & & 1
        \end{psmallmatrix}
        \begin{psmallmatrix}
            1 & m_3/m_2' & & \\
            & 1 & & \\
            & & 1 & -m_1/m_2' \\
            & & & 1
        \end{psmallmatrix}
        g) \, d^2u \\
        &= \int E(\chi,
        \begin{psmallmatrix}
            1 & & u_1 & x_2 \\
            & 1 & & u_3 \\
            & & 1 & \\
            & & & 1
        \end{psmallmatrix}
        \begin{psmallmatrix}
            1 & m_3/m_2' & & \\
            & 1 & & \\
            & & 1 & -m_1/m_2' \\
            & & & 1
        \end{psmallmatrix}
        g) \overline{\e(m_2' x_2)} \, dx_2 \, d^2 u \\
        &= \int E(\chi,
        \begin{psmallmatrix}
            1 & & u_1 & x_2 - \frac{m_1}{m_2'}u_1 - \frac{m_3}{m_2'}u_3\\
            & 1 & & u_3 \\
            & & 1 & \\
            & & & 1
        \end{psmallmatrix}
        g) \overline{\e(m_2' x_2)} \, dx_2 \, d^2 u \\
        &= \int E(\chi,
        \begin{psmallmatrix}
            1 & & u_1 & u_2 \\
            & 1 & & u_3 \\
            & & 1 & \\
            & & & 1
        \end{psmallmatrix}
        g) \overline{\e(m_1 u_1 + m_2' u_2 + m_3 u_3)} \, d^3u \\
        &= W_{N'}(\chi, m_1, m_2', m_3; g)
    \end{split}
\end{equation}

\vspace{1cm}

\section{Alternative expansions for \texorpdfstring{$F_{(21^2)}$}{F(211)}}
\label{app:SL4-F211-alternative-expansion}
Instead of expanding $F_{(21^2)}$ such that it is a sum of maximally degenerate Whittaker vectors charged on $\alpha_2$ in the minimal representation, we will now show that it is possible to let the maximally degenerate Whittaker vectors be charged only on $\alpha_1$ or $\alpha_3$ instead. 

The expansion made in the proof of theorem \ref{thm:SL4-min-rep} resembles more the proof in theorem \ref{thm:SL4-ntm-rep} than those presented here and was thus suited better as an introduction before the latter theorem. On the other hand the expansion carried out here is more easily generalised to $SL(n)$.

Let 
\begin{equation}
    l_1 = 
    \begin{psmallmatrix}
        1 & & & \\
        & & & 1 \\
        & & 1 & -m_2/m_1' \\
        & -1 & & m_3/m_1'
    \end{psmallmatrix}  
    \qquad
    l_2 =
    \begin{psmallmatrix}
        1 & & & \\
        & 1 & & \\
        & & 1 & \\
        & & -m_5/m_4' & 1
    \end{psmallmatrix}
    \qquad
    w =
    \begin{psmallmatrix}
        1 & & & \\
        & 1 & & \\
        & & & 1 \\
        & & -1 & 
    \end{psmallmatrix}
\end{equation}
where we recall that when summing over such $l_1$ and $l_2$ we may do a rescaling of the charges and consider them independent of $m_1'$ and $m_4'$.

Then,
\begin{equation}
    \begin{split}
        \MoveEqLeft
        F_{(21^2)}(\chi, m_1'; g) \\
        &= \sum_{m_2, m_3} \int E(\chi,
        \begin{psmallmatrix}
            1 & x_3 & x_2 & x_1 \\
            & 1 & & \\
            & & 1 & \\
            & & & 1
        \end{psmallmatrix} g) \overline{\e(m_1' x_1 + m_2 x_2 + m_3 x_3)} \, d^3x \\
        &=
        \sum_{l_1} \int E(\chi,
        \begin{psmallmatrix}
            1 & x_1 & x_2 & x_3 \\
            & 1 & & \\
            & & 1 & \\
            & & & 1
        \end{psmallmatrix} 
        l_1 g) \overline{\e(m_1' x_1)} \, d^3 x \\
        &= \sum_{\substack{l_1 \\ m_4, m_5}} \int E(\chi,
        \begin{psmallmatrix}
            1 & x_1 & x_2 & x_3 \\
            & 1 & x_4 & x_5 \\
            & & 1 & \\
            & & & 1
        \end{psmallmatrix} 
        l_1 g) \overline{\e(m_1' x_1 + m_4 x_4 + m_5 x_5)} \, d^5x \\
        &= \sum_{\substack{l_1, l_2 \\ m_4'}} \int E(\chi,
        \begin{psmallmatrix}
            1 & x_1 & x_2 & x_3 \\
            & 1 & x_4 & x_5 \\
            & & 1 & \\
            & & & 1
        \end{psmallmatrix}
        l_2 l_1 g) \overline{\e(m_1' x_1 + m_4' x_4)} \, d^5x +{} \\
        & \quad + \sum_{\substack{l_1 \\ m_5}} \int E(\chi,
        \begin{psmallmatrix}
            1 & x_1 & x_2 & x_3 \\
            & 1 & x_4 & x_5 \\
            & & 1 & \\
            & & & 1
        \end{psmallmatrix}
        w l_1 g) \overline{\e(m_1' x_1 + m_5 x_4)} \, d^5x \\
        &= \sum_{\substack{l_1, l_2 \\ m_4', m_6}} \int E(\chi,
        \begin{psmallmatrix}
            1 & x_1 & x_2 & x_3 \\
            & 1 & x_4 & x_5 \\
            & & 1 & x_6 \\
            & & & 1
        \end{psmallmatrix}
        l_2 l_1 g) \overline{\e(m_1' x_1 + m_4' x_4 + m_6 x_6)} \, d^6x +{} \\
        & \quad + \sum_{\substack{l_1 \\ m_5, m_6}} \int E(\chi,
        \begin{psmallmatrix}
            1 & x_1 & x_2 & x_3 \\
            & 1 & x_4 & x_5 \\
            & & 1 & x_6 \\
            & & & 1
        \end{psmallmatrix}
        w l_1 g) \overline{\e(m_1' x_1 + m_5 x_4 + m_6 x_6)} \, d^6x \\
    \end{split}
\end{equation}

In the minimal representation only the maximally degenerate Whittaker vector of the last term survives. Hence,
\begin{equation}
    F_{(21^2)}(\chi_\text{min}, m_1'; g) = \sum_{a, b} W_N(\chi_\text{min}, m_1', 0, 0;
    \begin{psmallmatrix}
        1 & & & \\
        & & & 1 \\
        & -1 & & a \\
        & & -1 & b
    \end{psmallmatrix} g)
\end{equation}
using table \ref{tab:SL4-Whittaker-as-orbits}.

A similar expansion can be done for $\alpha_3$ yielding
\begin{equation}
    F_{(21^2)}(\chi_\text{min}, m_6'; g) = \sum_{a, b} W_N(\chi_\text{min}, 0, 0, m_6';
    \begin{psmallmatrix}
        & -1 & & \\
        & & -1 & \\
        1 & a & b & \\
        & & & 1
    \end{psmallmatrix} g)
\end{equation}

\pagebreak

{\small 
\bibliography{Whittakerbib}
\bibliographystyle{utphys}
}

\end{document}